\documentclass{mcom-l}

\theoremstyle{thmstyleone}
\newtheorem{theorem}{Theorem}
\newtheorem{proposition}[theorem]{Proposition}
\theoremstyle{thmstyletwo}

\newtheorem{remark}{Remark}
\theoremstyle{thmstylethree}

\newtheorem{lemma}[theorem]{Lemma}
\newtheorem{corollary}{Corollary}
\numberwithin{equation}{section}
\usepackage{etex}
\usepackage{scrlayer-scrpage}
\usepackage[latin1]{inputenc}
\usepackage[T1]{fontenc}
\usepackage{amsfonts}
\usepackage{amssymb, amsmath}
\usepackage{geometry}
\usepackage{paralist}
\usepackage{array}
\usepackage{colortbl}
\usepackage{ragged2e}
\usepackage{lscape}
\usepackage[english]{babel} 
\usepackage{ifthen}
\usepackage{amsthm} 	
\usepackage{bbm}		
\usepackage{graphicx}
\usepackage{epstopdf}
\usepackage{lmodern}
\usepackage{float}
\usepackage{wrapfig} 
\usepackage{textcomp}
\usepackage{graphicx}
\usepackage{cancel}
\usepackage{mathrsfs}
\usepackage{upgreek}
\usepackage{textpos}
\usepackage{mathtools} 
\usepackage{ragged2e}
\usepackage{tabularx}
\usepackage{nicefrac}
\usepackage{multirow}
\usepackage{soul}
\usepackage{setspace}
\usepackage{yhmath}
\usepackage[symbol]{footmisc}

\def \J{\widehat{\;\textup{\textbf{J}}}_{\vec{n}} }
\def \indikator{\mathbbm{1}}
\def \indikatorfct{\chi}

\def \uIII{\widehat{\,\textup{\textbf{u}}_i}}
\def \zeros{\mathbb{O}}
\def \x{\textup{x}}
\def \P{\mathcal{P}}

\def \M{\mathcal{M}}

\def \D{\textup{D}}
\def \v{\widehat{\textup{\textbf{v}}}}
\def \T{\textup{T}}
\newcommand{\diag}{\textup{diag}}
\def \d{\,\textup{d}}
\DeclareMathOperator*{\argmin}{argmin}

\def \velocity{\boldsymbol{\nu}}
\def \velocityI{\boldsymbol{\nu}_1}
\def \velocityII{\boldsymbol{\nu}_2}
\def \velocityIII{\boldsymbol{\nu}_i}

\def \alphahat{\hat{\alpha}}

\def \veco{\vec{0}}
\def \DPk{d_k}
\def \VPk{v_k}
\def \VPl{v_\ell}
\def \VP{\mathcal{V}}
\def \VPhaar{\mathcal{H}}

\def \vhatStar{\boldsymbol{\widehat{v_*}}}
\def \DvhatStar{\boldsymbol{\widehat{\Delta v_*}}}

\def \DP{\mathcal{D}}
\def \entropy{\eta}

\def \AdmissibleSet{\mathbb{H}}
\def \squareRoot{\widehat{(\sqrt[2]{\u})}}
\def \nRoot{\widehat{(\sqrt[n]{\u})}}

\def \UNIT{\widehat{\;e_1}}

\def \RoeDetI{\alpha}
\def \RoeDetII{\beta}

\def \RoeII{\hat{\beta}}

\def \u{\textup{\textbf{u}}}
\def \uhat{\widehat{\textup{\textbf{u}}}}
\def \what{\widehat{\textup{\textbf{w}}}}
\def \yhat{\widehat{\textup{\textbf{y}}}}
\def \qhat{\widehat{\textup{\textbf{q}}}}

\def \uhati{\widehat{\,\textup{\textbf{u}}_i\,}}
\def \uhatj{\widehat{\,\textup{\textbf{u}}_j\,}}
\def \uhatI{\widehat{\,\textup{\textbf{u}}_1\,}}
\def \uhatII{\widehat{\,\textup{\textbf{u}}_2\,}}
\def \uhatd{\widehat{\,\textup{\textbf{u}}_d\,}}
\def \uI{\widehat{\,\textup{\textbf{u}}_i}}
\def \chat{\widehat{\textup{\textbf{c}}}}
\def \shat{\widehat{\textup{\textbf{s}}}}
\def \sign{\textup{sign}}

\begin{document}

\title[Haar-type stochastic Galerkin formulations]{Haar-type stochastic Galerkin formulations for hyperbolic systems with Lipschitz continuous flux function}

\author{Stephan Gerster}
\address{Universit\`{a} degli Studi dell'Insubria,  Como, Italy}
\email{stephan.gerster@gmail.com}

\author{Aleksey Sikstel}
\address{Technische Universität Darmstadt, Germany}
\email{sikstel@mathematik.tu-darmstadt.de}

\author{Giuseppe Visconti}
\address{Sapienza Universit\`{a} di Roma, Italy}
\email{giuseppe.visconti@uniroma1.it}

\date{}

\dedicatory{}

\begin{abstract}
	This work is devoted to the Galerkin projection of highly nonlinear random quantities. The dependency on a random input is described by Haar-type wavelet systems. The classical Haar sequence has been used by Pettersson, Iaccarino, Nordstr\"om (2014) for a hyperbolic stochastic Galerkin formulation of the one-dimensional Euler equations. This work generalizes their approach to several  multi-dimensional systems with Lipschitz continuous and  non-polynomial flux functions. Theoretical results are illustrated numerically by a genuinely multi-dimensional CWENO reconstruction.
\end{abstract}



\maketitle


\section{Introduction}
There are close relations between stochastic processes and orthogonal functions~\cite{S20}. 
Following early works by Wiener~\cite{S1}, where 
Hermite polynomials 
and homogeneous chaos has been used 
in the integration theory for Brownian motion, 
the representation of random quantities by orthogonal functions has 
received increasing attention~\cite{S2,S3,S4,S16}. 

In the context of intrusive stochastic Galerkin methods, 
the functional dependence on the stochastic input
is described a priori by a polynomial chaos~(gPC) expansion and a Galerkin projection is used to obtain
deterministic evolution equations for the coefficients in the series, called gPC modes. Then, all involved mathematical operations, e.g.~products and norms, must be adopted and  applied to the variables in the governing equations. 
This causes serious theoretical and numerical challenges~\cite{S15}. For instance, the accuracy of the truncated gPC representations affects the representation of  positive quantities, e.g.~the density of a gas, and nonlinearities may require a high gPC order, which  may be too computationally expensive.

To this end,  wavelet-based  expansions have been introduced~\cite{LEMAITRE2004,S5,S9}.  
These are motivated by a robust expansion for solutions that depend on the stochastic input in a non-smooth way and are used for stochastic multiresolution as well as adaptivity~\cite{S9,SI1,Kroeker2015,Tryoen2012,Igor}. 

The stochastic Galerkin method has been  applied  to diffusion~\cite{Webster2016,P1,P2,Igor}, kinetic~\cite{K1,K2,Yuhua2017,Yuhua2018,Hui2020} and Hamilton-Jacobi equations~\cite{HamiltonJacobiJin,LevelSet}. 	
In general, results for hyperbolic systems are not available~\cite{H0,S4}, since desired properties like hyperbolicity and the existence of entropies are not transferred to the intrusive  formulation. A problem is posed by the fact that the deterministic Jacobian of the projected system differs from the random Jacobian of the original system and, therefore, not even real eigenvalues, which are necessary for hyperbolicity,  are  guaranteed in general.

There are many examples that show a loss of hyperbolicity, when the stochastic Galerkin approach is applied directly to hyperbolic systems. 
To this end, auxiliary variables have been introduced to establish wellposedness results. For instance, entropy-entropy flux pairs can be obtained by an expansion in entropy variables, i.e.~the gradient of the deterministic entropy~\cite{H0,StephansDiss,kusch2019maximum,kusch2020filtered,kusch2020intrusive}.  
Roe variables, which include the square root of the density, preserve hyperbolicity for Euler equations~\cite{Roe1,S5,FettesPaper,GersterHertyCicip2020}. 
The drawback of introducing auxiliary variables is an additional computational overhead that arises from an optimization problem, which is required to calculate the auxiliary variables. 
Furthermore, the transforms may  involve integrals and derivatives 
that have to be calculated numerically~\cite[Sec.~3.3]{StephansDiss}. But in this case,
the formulation is not \emph{truly intrusive},  i.e.~not all integrals are computed \emph{exactly}. 
Another problem is the connection between the hyperbolicity of the Galerkin and the original system. Namely, the required optimization problems are well-posed only in a possibly small domain~\cite[Sec.~4]{FettesPaper} when the solvability of the transform is guaranteed only locally by the implicit function theorem.

Recently, a hyperbolic stochastic Galerkin formulation for  the shallow water equations has been presented that  neither requires  auxiliary variables nor any transform, since the  
Jacobian is shown to be similar to a symmetric 
matrix~\cite{Epshteyn2021,Epshteyn2022}. 
Then, positivity of the water height at a finite 
number of stochastic quadrature points is sufficient to preserve hyperbolicity of the stochastic Galerkin formulation.

These results exploit \emph{quadratic} relationships that are expressed efficiently by the Galerkin product. 
Extensions of the classical polynomial chaos expansions to general nonlinearities, e.g.~with Legendre or Hermite polynomials, however, are not straightforward. Inspired by the application of the Haar sequence to solve Euler equations~\cite{S5}, we express nonlinearities, which occur  for instance in isentropic Euler and level set equations, by  wavelet families that are generated by Haar-type matrices~\cite{Resnikoff1998}. Those are
widely used in signal processing~\cite{Pratt2007,Resnikoff1998}, e.g.~for the Hadamard,  Walsh, Chebyshev matrices and for the discrete cosine transform. 

The aim of the paper is to show that many theoretical and numerical challenges that occur for polynomial expansions vanish when   these Haar-type expansions are employed.  
More precisely, our contribution is summarized as follows:
\begin{itemize}
	\item 
Firstly, stochastic Galerkin systems presented in this manuscript have not been proven hyperbolic for general basis functions so far (with the exception of level set equations)~\cite{LevelSet}.
	\item 
Secondly, the formulations require no optimization problems in contrast to previous works~\cite{S5,FettesPaper,LevelSet}, where  Roe-type approaches are used.
	\item 
Thirdly, the main Theorem~\ref{LemmaHaar} presents novel calculation rules that are important for other works, e.g.~\cite{Janina},  where an entropy-stable and well-balanced discontinuous Galerkin method is derived.
\end{itemize}

In particular, wellposedness results are achieved for hyperbolic conservation laws with Lipschitz continuous flux function.  
Since all expressions are stated in closed form without any optimization problem, 
the~connection between the hyperbolicity of the Galerkin and the original system is established. Namely, deterministic wellposedness properties of the original system at stochastic quadrature points transfer wellposedness results to the stochastic Galerkin formulation. 
Furthermore,  all eigenvalue decompositions are stated in closed form which allows for an efficient implementation also in higher dimensions. 
The presented results are~\emph{not} restricted to a dimensional splitting. Namely, the characteristic speeds are real with respect to \emph{all} spatial dimensions. This enables genuinely multi-dimensional high order methods based on the CWENO reconstruction in space~\cite{SempliceCocoRusso,CastroSemplice2019} that also allows for discretizations of balance laws. 

This paper is structured as follows. Stochastic Galerkin matrices are introduced in Section~\ref{SectionSGmatrices} for general polynomial chaos expansions. 
Section~\ref{SectionHaar} describes the representation of nonlinear quantities by Haar-type expansions. In particular, Section~\ref{SectionChallenges} addresses analytical and numerical challenges that   occur when these special expansions are not used. 
Those are applied  to stochastic Galerkin formulations for   hyperbolic conservation laws in Section~\ref{SectionConservationLaws}. 
Finally, Section~\ref{Numerics} presents numerical experiments that confirm the theoretical findings.

\section{General stochastic Galerkin formulations}\label{SectionSGmatrices}

\noindent
We introduce a random variable $\xi \, :\, \Omega\rightarrow \mathbb{R}$,  which is defined on a probability space~$\big(\Omega,\mathcal{F}(\Omega),\mathbb{P}\big)$, and associated orthonormal basis functions $\big\{ \phi_k(\xi) \big\}_{k\in\mathbb{N}_0}$ satisfying
$$
\mathbb{E}\Big[
\phi_i(\xi) \phi_j(\xi)
\Big]
=
\int \phi_i(\xi) \phi_j(\xi) \d \mathbb{P}
\eqqcolon
\langle \phi_{i}, \phi_{j}  \rangle_{\mathbb{P}}
=
\delta_{i,j},
$$
where~$\delta_{i,j}$ denotes the Kronecker delta. 
Then, for each random, scalar, square-integrable state $\u(\xi)\in\mathbb{L}^2(\Omega,\mathbb{P})$ there exist \textbf{gPC modes} $\uhat\in\mathbb{R}^{K+1}$ with~${K\in\mathbb{N}_0}$ such that the \emph{truncated} series
\begin{equation}\label{gPCseries}
	\Pi_K\big[\uhat\big](\xi)
	\coloneqq
	\sum\limits_{k=0}^K
	\uhat_k \phi_k(\xi)
	\quad\text{satisfies}\quad
	\Big\lVert 
	\Pi_K\big[\uhat\big](\xi) 
	-
	\u(\xi)
	\Big\rVert_{\mathbb{P}}
	\overset{K\rightarrow \infty}{\longrightarrow}0,
\end{equation}
provided that the probability measure satisfies mild conditions~\cite{S2,funaro2008,Ullmann2010}.  
\textbf{Stochastic Galerkin matrices}~\cite{S15,S4,S18,Ullmann2010} which are defined by
$$
\P(\widehat{\u})\coloneqq\sum\limits_{k=0}^K \widehat{\u}_{k}\mathcal{M}_{k}
\quad\text{and}\quad
\mathcal{M}_{k}\coloneqq
\Big(
\langle \phi_{k},\phi_{i}\phi_{j} \rangle_{\mathbb{P}}\Big)_{i,j=0,\ldots,K}
$$
arise in a Galerkin projection of a random product~$\u(\xi)\textbf{q}(\xi)\in\mathbb{L}^2(\Omega,\mathbb{P})$, called \textbf{Galerkin product}, that is denoted by
$$\widehat{\u}
\ast
\widehat{\textbf{q}}
\coloneqq
\P(\widehat{\u})
\widehat{\textbf{q}}
\ \
\text{and satisfies}
\ \
\Big\lVert 
\Pi_K\big[\widehat{\u}
\ast
\widehat{\textbf{q}}\big](\xi) 
-
\u(\xi)\textbf{q}(\xi)
\Big\rVert_{\mathbb{P}}
\rightarrow 0
\ \ \text{for}\ \
K\rightarrow \infty. 
$$

\noindent
Due to the orthonormality of the basis functions, we obtain the identity matrix by~$\P(\UNIT)=\M_0=\indikator\in
\mathbb{R}^{(K+1)\times (K+1)}
$, where~$\UNIT\in\mathbb{R}^{K+1}$ denotes the unit vector. 
Furthermore, the choice 
$$
\uhat^{\ast m}
\coloneqq
\P^{m}(\uhat) \UNIT
=
\P^{m-1}(\uhat) \P(\UNIT) \uhat
=
\uhat
\ast \ldots \ast \big( \uhat \ast (\uhat\ast\uhat) \big)
$$ 	
with~$m\in\mathbb{N}$ and~$\uhat^{\ast 0} \coloneqq \UNIT$  
is a consistent approximation of the $m$-th power~$\u^m$ of a random function~$\u(\xi)$, i.e.
$$
\Big\lVert
\u^m(\xi)
-
\Pi_K\left[		\uhat^{\ast m} \right](\xi)
\Big\rVert_{\mathbb{P}}
\rightarrow
0
\quad\text{for}\quad
K\rightarrow\infty.
$$

\noindent
The matrix~$\P(\uhat)$ defines a linear operator. Hence, the following  sets are convex. 
\begin{equation}\label{Definitionsbereich}
	\begin{aligned}
		\mathbb{H}^+
		&\coloneqq 
		\Big\{ \, \widehat{\u} \in \mathbb{R}^{K+1} \ \Big\vert \ \P(\widehat{\u}) \text{ is strictly positive definite} \Big\},\\
		\mathbb{H}^+_0
		&\coloneqq 
		\Big\{ \, \widehat{\u} \in \mathbb{R}^{K+1} \ \Big\vert \ \P(\widehat{\u}) \text{ is positive semidefinite} \Big\}.
	\end{aligned}
\end{equation}
These positive definiteness assumptions reflect the representation of positive quantities in the truncated expansion~\eqref{gPCseries}. More precisely, the symmetric matrix~$\P(\widehat{\u})$  is strictly positive definite if all   realizations~$\Pi_K\big[\uhat\big]\big(\xi(\omega)\big)>0$ are $\mathbb{P}$-almost surely ($\mathbb{P}$-a.s.) positive~\cite{Sonday2011,H3}. 
It  admits an orthogonal eigenvalue decomposition that is denoted by~$
\P(\widehat{\u})
=
\VP(\widehat{\u})
\DP(\widehat{\u})
\VP^\T(\widehat{\u})
$ 
and satisfies the relation~$\VP^\T(\widehat{\u})=\VP^{-1}(\widehat{\u})$. Furthermore, the Galerkin product is symmetric, i.e.~$\widehat{\u}
\ast
\qhat
=
\qhat
\ast
\widehat{\u}$. 
For general basis functions we have the following lemma. 

\begin{lemma}\label{Lemma0}	
	Consider the orthogonal eigenvalue decomposition of the stochastic Galerkin matrix ${
		\P(\widehat{\u})
		=
		\VP(\widehat{\u})
		\DP(\widehat{\u})
		\VP^\T(\widehat{\u})
	}$ 
	with eigenvectors 
	$\VP = \big(v_0\vert\cdots\vert v_K\big)$ 
	and eigenvalues 
	$\DP = \diag\big\{d_0,\ldots,d_K\big\}$ 
	that satisfy the relation~$	
	\P(\widehat{\u})
	v_k(\widehat{\u})
	=
	d_k(\widehat{\u})
	v_k(\widehat{\u})
	$. 	
	Then, we have the equality\footnote{The Jacobian  is denoted by~$\D_{\bar{\alpha}}$ and the gradient by~$\nabla_{\bar{\alpha}}$.} 
	$$
	\D_{\bar{\alpha}}
	\DP(\bar{\alpha})\big\vert_{\bar{\alpha}=\widehat{\u}} \VP^\T(\widehat{\u})\qhat
	=
	\VP^\T(\widehat{\u})
	\P(\qhat)
	\quad\text{for all}\quad
	\widehat{\u}, \qhat \in\mathbb{R}^{K+1}.
	$$	
	
\end{lemma}

\noindent
The proof follows basic algebraic arguments and is given in the appendix.


\section{Haar-type stochastic Galerkin formulations}\label{SectionHaar}

\noindent
It has been shown in~\cite[Appendix~B]{S5} that
the \textbf{Haar sequence}~\cite{Haar1910} yields a stochastic Galerkin matrix with constant eigenvectors, i.e.~${\VP(\hat{\alpha})= \VPhaar}$ for all~$\hat{\alpha}\in\mathbb{R}^{K+1}$. More precisely, it is defined on a  level~${J\in\mathbb{N}_0}$ and generates a  \textbf{wavelet system} with ${K+1=2^{J+1}}$ elements by 
\begin{align*}
	&\begin{aligned}
		\mathbb{W}\big[\VPhaar_J\big] &\coloneqq \big\{ \indikatorfct_{[0,1]}(\xi),  \psi_{\boldsymbol{\ell}}(\xi) \ \big\vert \ 
		\boldsymbol{\ell} \in \mathbb{L}_j
		\big\} 
		\qquad \text{with piecewise constant functions} \  \\
		\psi_{\boldsymbol{0}}(\xi) &\coloneqq
		\begin{cases}
			1  & \text{if \ } 0 \leq \xi < \nicefrac{1}{2},  \\
			-1 & \text{if \ } \nicefrac{1}{2} \leq \xi < 1,  \\
			0  & \text{else,}
		\end{cases} \quad
		\psi_{\boldsymbol{\ell}}(\xi)
		\coloneqq
		2^{\nicefrac{j}{2}}
		\psi_{\boldsymbol{0}}
		\big(2^j \xi - k \big) 
		\text{ \ for \ }
		\boldsymbol{\ell} \in\mathbb{L}_J \backslash \{ \boldsymbol{0} \}
	\end{aligned}\\
	& \text{and index set}\quad
	\mathbb{L}_J \coloneqq 
	\big\{ \boldsymbol{0} \big\} \cup
	\big\{ (k,\ell) \ \big\vert \ k=0,\ldots,2^j-1, \ j=1,\ldots,J  
	\big\},
\end{align*}
where~$\indikatorfct_{[0,1]}$ denotes the indicator function on~$[0,1]$. 
Using a lexicographical order for the bases functions, the resulting stochastic Galerkin matrix reads as
\begin{equation}\label{classicalSGmatrix}
	\P_J(\widehat{\u})\coloneqq\sum\limits_{\boldsymbol{\ell}\in \mathbb{L}_J } \widehat{\u}_{\boldsymbol{\ell}}\mathcal{M}_{\boldsymbol{\ell}}
	\quad\text{and}\quad
	\mathcal{M}_{\boldsymbol{\ell}}\coloneqq
	\Big(
	\langle \psi_{\boldsymbol{\ell}},\psi_{\boldsymbol{i}}\psi_{\boldsymbol{j}} \rangle_{\mathbb{P}}\Big)_{\boldsymbol{i},\boldsymbol{j}\in\mathbb{L}_J}. 
\end{equation}
According to~\cite{S5,Pratt2007}, it admits the orthogonal eigenvalue decomposition~$
\P_J({\uhat}) = \VPhaar_J \mathcal{D}_J({\uhat}) \VPhaar_J^{\T}
$ 
with the \textbf{classical Haar matrix} 
\begin{equation}\label{HaarTrafo}
	\VPhaar_{J}
	=
	\frac{1}{\sqrt{2}}
	\begin{pmatrix}
		\VPhaar_{J-1} \otimes (1,1) \\
		\indikator \otimes (1,-1)
	\end{pmatrix}
	\quad\text{for}\quad
	\VPhaar_{0} =
		\frac{1}{\sqrt{2}}
	\begin{pmatrix}
		1 & 1 \\ 1 & -1
	\end{pmatrix},
\end{equation}
where~$\otimes$ denotes the Kronecker product and~$\indikator\coloneqq \diag\{1,\ldots,1\}$ the identity matrix. 
More generally, we consider \textbf{Haar-type matrices} 
$$
\VPhaar
=
\begin{pmatrix}
	1 & \veco^{\,\T} \\
	\veco & \mathcal{O}_{\VPhaar}
\end{pmatrix}
\VPhaar_c
\quad\text{for}\quad
\VPhaar_c
= 
\begin{pmatrix}
	1 & 1 & \cdots & 1 \\
	h^d_{1}  & h^r_1 & \ldots &  h^r_1 \\
	& \ddots & \ddots & \vdots \\
	& & h^d_{K}  &  h^r_K
\end{pmatrix},
\quad
\begin{aligned}
	h^d_i
	&\coloneqq
	s_i h^r_i , \\
	h^r_i
	&\coloneqq
	\sqrt{ \frac{K+1}{s_i^2+s_i} }, \\
	s_{i}
	&\coloneqq
	-(K+1-i) \\
\end{aligned}
$$
with $\veco\coloneqq (0,\ldots,0)^\T$. Here,~$\VPhaar_c\in\mathbb{R}^{(K+1)\times (K+1)}$ is the \textbf{canonical Haar matrix}~\cite[Th.~4.1, Cor.~4.4]{Resnikoff1998} and~$\mathcal{O}_{\VPhaar}\in\mathbb{R}^{K\times K}$ is an orthogonal matrix. 
All of these Haar-type matrices~$\VPhaar$ generate in turn a wavelet sytem
\begin{equation}\label{WaveletSystem}
	\begin{aligned}
		\mathbb{W}\big[\VPhaar\big] &\coloneqq \Big\{ \indikatorfct_{[0,1]}(\xi), \ 
		\phi_1(\xi),\ldots,
		\phi_K(\xi) 
		\Big\}
		\quad\text{for}\quad \\
		\phi_k(\xi) 
		&\coloneqq
		\sum\limits_{\ell=0}^K  \VPhaar_{k,\ell} \, \indikatorfct_{[0,1]}
		\big(
		(K+1)\xi - \ell
		\big)
	\end{aligned}
\end{equation}
that allows to approximate any square-integrable function~\cite[Th.~5.1,~Th.~5.2]{Resnikoff1998}. 
Many classical transforms used in signal processing are Haar matrices~\cite[Ch.~8]{Pratt2007}. These include~Hadamard,  Walsh,  Chebyshev matrices and the \textbf{discrete cosine transform (DCT)} matrix~\cite[Ex.~4.7]{Resnikoff1998} with entries
\begin{equation}\label{CosineTrafo}
	\left(\VPhaar_{\cos}\right)_{i,j}
	=
	\begin{cases}
		1 
		& \text{if \ \ } i=1, \\
		\sqrt{2}\cos\Big(
		\frac{(i+1) ( 2j-1 ) }{2(K+1)}
		\Big)
		& \text{if \ \ } i>1.
	\end{cases}
\end{equation}

\begin{remark}
	The drawback of piecewise constant expansions is the lack of spectral convergence. 
	Whether a given basis leads to an eigenvalue decomposition with constant eigenvectors can be verified numerically by checking if all precomputed matrices~$\M_k$  commute~\cite[Lem.~4.1]{GersterHertyCicip2020}. This allows also for an extension to piecewise linear functions.  More precisely, the domain~$\Xi\coloneqq[0,1]$ can be partitioned into~$N\in\mathbb{N}$ subdomains
	$$
	\Xi_k\coloneqq(\xi_k,\xi_{k+1}) 
	\quad\text{with}\quad
	\xi_k\coloneqq \nicefrac{k}{N}
	\quad\text{and}\quad
	k\in\mathbb{J}\coloneqq\{ 0,\ldots,N-1 \}.
	$$
	The $L^2(\Xi)$-scaled orthogonal piecewise functions are
	$$
	\begin{aligned}
		\mathbb{W}_{\Xi}
		&\coloneqq
		\bigcup_{k\in\mathbb{J}}
		\mathbb{W}_{\Xi_k}
		\quad\text{for}\\
		\mathbb{W}_{\Xi_k}
		&\coloneqq
		\Big\{ \psi_{k,0}(\xi) \coloneqq
		\indikatorfct_{[\xi_k,\xi_{k+1}]}(\xi)
		\sqrt{N} , \ \psi_{k,1}(\xi) \coloneqq 
		\indikatorfct_{[\xi_k,\xi_{k+1}]} (\xi)
		\sqrt{3N}  \xi  \Big\}.
	\end{aligned}
	$$
	Since the supports of~$\psi_{k,i}$ and $\psi_{\ell,i}$ are disjoint for $k\neq\ell$, the stochastic Galerkin matrix has the block diagonal structure
	\begin{align*}
		&\P_{\Xi}(\uhat)
		\coloneqq
		\sum\limits_{\boldsymbol{\ell}\in\mathbb{L}_{\Xi}}
		\uhat_{\boldsymbol{\ell}} \M_{\boldsymbol{\ell}}
		=
		\diag\bigg\{
		\sum\limits_{i=0,1}
		\uhat_{0,i}\,
		\M^{(i)},
		\ldots,
		\sum\limits_{i=0,1}
		\uhat_{N,i}\,
		\M^{(i)}
		\bigg\} \\
		&\text{for}\quad
		\mathcal{M}_{\boldsymbol{\ell}}\coloneqq
		\Big(
		\langle \psi_{\boldsymbol{\ell}},\psi_{\boldsymbol{i}}\psi_{\boldsymbol{j}} \rangle_{\mathbb{P}}\Big)_{\boldsymbol{i},\boldsymbol{j}\in\mathbb{L}_{\Xi}}
		\quad\text{and}\quad
		\M^{(i)} 
		\coloneqq
		\Big(
		\langle \psi_{k,m} \psi_{k,n} \psi_{k,i} \rangle_{\mathbb{P}}\Big)_{m,n\in\{0,1\}}
	\end{align*}
	with index set~$
	\mathbb{L}_{\Xi}
	\coloneqq
	\big\{
	(k,i) \ \big\vert \ k\in\mathbb{J}, \ i=0,1
	\big\}
	$. 
	Hence, the matrices~$\M_{\boldsymbol{\ell}} \in\mathbb{R}^{2N\times 2N}$ commute, since the matrix~$\M^{(1)} \in\mathbb{R}^{2\times 2}$ commutes with the identity matrix~$\M^{(0)}=\indikator \in\mathbb{R}^{2\times 2}$. \\
	
\end{remark}

\noindent
For general bases, e.g.~Hermite or Legendre polynomials, the stochastic Galerkin matrices do not commute~\cite{S4}. 
When Haar-type expansions are used, however, we have
\begin{equation}\label{Commutativity}
	\P(\uhat)\P(\widehat{\textbf{w}}) 
	=
	\VPhaar\big[ \DP(\uhat)\DP(\widehat{\textbf{w}}) \big]\VPhaar^\T
	=
	\VPhaar\big[\DP(\widehat{\textbf{w}}) \DP(\uhat) \big] \VPhaar^\T 
	=
	\P(\widehat{\textbf{w}}) \P(\uhat).
\end{equation}

\noindent 
The commutativity~\eqref{Commutativity}  yields  two important properties that 
\textbf{Haar-type expansions} fulfill, which are stated in the following lemma.

\begin{lemma}\label{Lemma2}	Given a Haar-type expansion the following statements hold.
	\begin{enumerate}[(i)]
		\item
		The equality~$
		\P\big(
		\P(\uhat)
		\what
		\big)
		=
		\P(\uhat)
		\P(
		\what
		)
		$ 
		is satisfied  
		for all 
		$\uhat, \what \in\mathbb{R}^{K+1}$.
		\item
		The mapping
		$		
		\mathcal{T}_n \;  : \;  \AdmissibleSet_0^+\rightarrow\AdmissibleSet^+_0,
		\ 
		\uhat \mapsto \uhat^{\ast n} =\P^{n}(\uhat) \UNIT
		$ is bijective for all~${n\in\mathbb{N}}$.  Its inverse~reads as~$
		\mathcal{T}_n^{-1}(\uhat)
		=
		\VPhaar\mathcal{D}^{ \nicefrac{1}{n}} \big(
		\uhat
		\big) \VPhaar^\T \UNIT$. 
	\end{enumerate}
\end{lemma}

\begin{proof}
	Equation~\eqref{Commutativity} and the symmetry of the Galerkin product  imply 
	the first statement, i.e.
	$$
	\P\Big( \P(\uhat) \what \Big)
	\yhat 
	= 
	\P\big(\yhat) \Big( \P(\uhat \big) \what \Big)
	=
	\P\big(\uhat\big) \P\big(\yhat\big) \what 
	=
	\P\big(\uhat\big) \P\big(\what\big) \yhat 
	\quad \text{for all}\quad
	\yhat\in\mathbb{R}^{K+1}. 	
	$$
	
	\noindent
	Statement~\textit{(i)}  implies
	\begin{equation}\label{Lemma2Eq1}
		\P\big( \uhat^{\ast n} \big)= \P\Big( \P^{n}(\uhat)\UNIT \Big)
		=
		\P^{n}(\uhat)
		\quad\Leftrightarrow\quad
		\mathcal{D}\big(\uhat^{\ast n}\big)
		=
		\mathcal{D}^n(\uhat).
	\end{equation}
	Hence, the mapping is bijective and the inverse~$\mathcal{T}_n^{-1}(\uhat)$ is obtained by
	\begin{align}
		&\uhat^{\ast n}
		=
		\P^{n}(\uhat)\UNIT
		=
		\P^{n-1}(\uhat)\uhat
		=
		\VPhaar\mathcal{D}^{n-1}(\uhat) \VPhaar^\T \uhat
		=
		\VPhaar\mathcal{D}^{ 1-\frac{1}{n} } \big(\uhat^{\ast n}\big) \VPhaar^\T \uhat \nonumber \\
		&\Leftrightarrow\quad
		\uhat
		=
		\VPhaar\mathcal{D}^{ \frac{1}{n}-1 } \big(
		\uhat^{\ast n}
		\big) \VPhaar^\T \uhat^{\ast n}
		= \mathcal{T}^{-1}_n\big(\uhat^{\ast n}\big).\label{Lemma2Eq2}
	\end{align}
\end{proof}

\noindent
The following theorem states gPC modes for several non-polynomial random quantities.

\begin{theorem}\label{LemmaHaar}
	Let a Haar-type gPC expansion be given. Then, the following statements are satisfied.
	\begin{enumerate}[(i)]
		\item
		The gPC modes of the function~$p(\u)\coloneqq\u^\gamma$ for the exponent $\gamma\in\mathbb{Q}^+_0$ and with the positive expansion~$\Pi_K\big[	\uhat\big]\big(\xi(\omega)\big) \in \mathbb{R}^+_0$ read as
		$$
		\widehat{p}(\uhat)
		=
		\VPhaar \DP(\uhat)^\gamma \VPhaar^\T \UNIT
		\quad\text{with Jacobian}\quad
		\D_{\uhat} \widehat{p}(\uhat)
		=
		\gamma
		\VPhaar \DP(\uhat)^{\gamma-1} \VPhaar^\T.
		$$
		\item 
		The gPC modes for the signum function of the real-valued expansion~$\Pi_K\big[	\uhat\big]\big(\xi(\omega)\big) \in \mathbb{R}$ are
		$$
		\shat(\uhat)
		\coloneqq 
		\VPhaar \sign\big\{ \DP(\uhat) \big\} \VPhaar^\T\UNIT.
		$$
		
		\item
		The gPC modes of the absolute value of the real-valued expansion~$\Pi_K\big[	\uhat\big]\big(\xi(\omega)\big) \in \mathbb{R}$ read as
		$$
		\widehat{\,\vert\u\vert\,}
		=
		\shat(\uhat) \ast \uhat
		\quad
		\text{with Jacobian}
		\quad
		\D_{\uhat} \widehat{\,\vert\u\vert\,}
		=
		\VPhaar \sign\big\{ \DP(\uhat) \big\} \VPhaar^\T.
		$$
		\item	 
		The gPC modes of the  $p$-norm for a vector-valued expansion~$\Pi_K\big[	\uhat\big]\big(\xi(\omega)\big) \in \mathbb{R}^d$ with gPC coefficients~$\uhat=(\uhatI,\ldots,\uhatd)^\T$ are
		\begin{align*}
			&\widehat{\,\lVert \u \rVert_1}
			=
			\sum\limits_{i=1}^d \VPhaar \big\vert \DP(\uI) \big\vert \VPhaar^\T \UNIT
			\quad\text{and}\quad
			\widehat{\,\lVert \u \rVert_p}
			=
			\VPhaar \Big\vert \DP\big(\chat(\uhat)\big) \Big\vert^{\frac{1}{p}} \VPhaar^\T \UNIT \\
			&\text{for}\quad
			\chat(\uhat)
			\coloneqq
			\sum\limits_{i=1}^d
			\widehat{\,\vert\u_i\vert\,}^{\ast p}
			\quad\text{with}\quad
			\widehat{\,\vert\u_i\vert\,}^{\ast p}
			=
			\VPhaar \big\vert \DP(\uI) \big\vert^p \VPhaar^\T \UNIT.
		\end{align*}
		The Jacobian reads as
		$$
		\D_{\uI} \widehat{\,\lVert \u \rVert_p}
		=
		\VPhaar 
		\Big[ 
		\DP^{\frac{1}{p}-1} \big(\chat(\uhat)\big) 
		\big\vert\DP (\uhati)\big\vert^{p-1} 
		\DP\big( \shat(\uhati) \big)
		\Big]
		\VPhaar^\T.
		$$
	\end{enumerate}
	
\end{theorem}

\begin{proof} \textit{Statement (i):}
	According to~\cite[Th.~2]{Sonday2011} and \cite[Th.~2.1]{H3}, it holds~$\uhat\in\mathbb{H}^+_0$, i.e.~$\DP(\uhat)\geq 0$, provided that 
	all realizations~$\u\big(\xi(\omega)\big) \in \mathbb{R}^+_0$  are non-negative. Due to Lemma~\ref{Lemma0} and~Lemma~\ref{Lemma2}, which implies~$
	\D_{\uhat}  \DP(\uhat) \VPhaar^\T \UNIT
	=\VPhaar^\T
	$, we have for the exponent~$\gamma=\nicefrac{m}{n}\in\mathbb{Q}^+_0$ and arbitrary~$m\in\mathbb{N}_0$, $n\in\mathbb{N}$  the gPC modes
	\begin{align*}
		&\widehat{p}(\uhat) 
		=
		\mathcal{T}_m\Big[ \mathcal{T}^{-1}_n  \big[
		\uhat
		\big]\Big]  
		=
		\VPhaar \DP(\uhat)^\gamma \VPhaar^\T \UNIT, \\
		&\text{with Jacobian}\quad
		\D_{\uhat} \widehat{p}(\uhat) 
		=
		\gamma
		\VPhaar  \DP(\uhat)^{\gamma-1}  \D_{\uhat}  \DP(\uhat) \VPhaar^\T \UNIT
		=
		\gamma
		\VPhaar  \DP(\uhat)^{\gamma-1} \VPhaar^\T.
	\end{align*}

	
	\noindent
	\textit{Statement (iii):} 
	Due to~$\mathcal{T}_2[\widehat{\u}]\in \mathbb{H}^+_0$ for all~$\widehat{\u}\in\mathbb{R}^{K+1}$, the gPC modes for the absolute value read as 
	$$
	\widehat{\,\vert\u\vert\,}=
	\mathcal{T}^{-1}_2\Big[ \mathcal{T}_2  \big[
	\uhat
	\big]\Big]
	=
	\VPhaar
	\sqrt{
		\DP^{ 2 } (\uhat)}\VPhaar^\T \UNIT
	=
	\VPhaar
	\big\vert
	\DP (\uhat) \big\vert \VPhaar^\T \UNIT.
	$$
	\textit{Statement~(ii):} This statement follows from
	\begin{align*}
		&\widehat{\,\vert\u\vert\,}
		=
		\VPhaar\Big[
		\sign\big\{ \DP(\uhat) \big\} \DP (\uhat) \Big] \VPhaar^\T \UNIT
		=
		\shat(\uhat) \ast \uhat 
		\quad\text{and}\\
		&\Big\lVert \Pi_K\big[\shat(\uhat) \ast \uhat \big]
		-
		\sign\big\{ \u(\xi) \big\} \u(\xi)
		\Big\rVert_{\mathbb{P}}
		\rightarrow 0
		\quad\text{for}\quad
		K\rightarrow \infty. 
	\end{align*}
	
	\noindent
	\textit{Statement (iv):} 
	For a~$d$-dimensional random quantitiy~$\u(\xi) = \big( \u_1(\xi),\ldots,\u_d(\xi)\big)^\T$, we obtain the gPC modes for the $p$-norm by
	$$
	\widehat{\ \lVert\u\rVert_p\,}
	=
	\mathcal{T}^{-1}_p\Bigg[
	\sum\limits_{i=1}^d
	\widehat{\,\vert\u_i\vert\,}^{\ast p}
	\Bigg]
	\ \ \text{with}\ \
	\widehat{\,\vert\u_i\vert\,}^{\ast p}
	=
	\P^p\Big(
	\widehat{\,\vert\u_i\vert\,}
	\Big)
	\UNIT
	=
	\VPhaar
	\big\vert
	\DP (\uhati) \big\vert^p \VPhaar^\T \UNIT \in \mathbb{H}^+_0.
	$$
	Lemma~\ref{Lemma0} yields~$\D_{\uhat} \widehat{\,\vert\u\vert\,}
	=\VPhaar
	\D_{\alphahat} \DP(\alphahat)\big\vert_{\alphahat=\uhat} 
	\VPhaar^\T\shat(\uhat) 
	=\P\big(\shat(\uhat)\big)
	$. 
	Hence, we have
	\begin{align*}
		\D_{\uhat} 
		\widehat{\,\vert\u\vert\,}^{\ast p}
		&= 
		p \VPhaar \big\vert\DP(\uhat)\big\vert^{p-1} 
		\D_{\uhat} \big\vert\DP(\uhat)\big\vert 
		\VPhaar^\T\UNIT\\
		&=
		p \VPhaar \big\vert\DP(\uhat)\big\vert^{p-1} \VPhaar^\T
		\D_{\uhat} \widehat{\,\vert\u\vert\,}\\
		&=
		\VPhaar\Big[
		p  \big\vert\DP(\uhat)\big\vert^{p-1} \DP\big(\shat(\uI) \big)
		\Big]
		\VPhaar^\T. 
	\end{align*}
	Then, the chain rule implies 
	\begin{align*}
		\D_{\uI} 
		\widehat{\ \lVert\u\rVert_p\,}
		&=
		\VPhaar \D_{\chat} \Big[
		\DP^{\frac{1}{p}} \big(\chat(\uhat)\big) \UNIT
		\Big] \D_{\uI} \widehat{\, \vert\u_i\vert \,}^{\ast p} \\
		&=
		\VPhaar \Big[ 
		\DP^{\frac{1}{p}-1} \big(\chat(\uhat)\big) 
		\big\vert\DP (\uhati)\big\vert^{p-1} 
		\DP\big( \shat(\uhati) \big)
		\Big]
		\VPhaar^\T.
	\end{align*}
\end{proof}

\noindent
We emphasize that all gPC modes, which are presented in~Theorem~\ref{LemmaHaar}, are stated in closed form. More precisely, they only depend on the eigenvalues~$\DP(\uhat) = \VPhaar^\T\P(\uhat) \VPhaar$ that are directly obtained, since the Haar-type matrix~$\VPhaar$ and the linear operator~$\P(\uhat)$ are known.

\section{Comparison between Haar-type and general polynomial chaos expansions}\label{SectionChallenges}
It has been emphasised in \cite{S15} that the representation of non-polynomial quantities by polynomial chaos expansions is challenging. 
In particular, one may define the square root to gPC modes~$\uhat\in\mathbb{R}^{K+1}$ as a solution to the nonlinear system~$\P(\alphahat)\alphahat =\uhat$. However, the solution is not necessarily unique.  To this end, it has been proposed in~\cite[Lem.~3.1]{GersterHertyCicip2020} to identify the square root as the  minimum 
\begin{align*}
	&	\squareRoot
	\coloneqq \argmin\limits_{\alphahat \in \AdmissibleSet^+} \Big\{ \entropy^{(2)}_{\uhat}(\alphahat) \Big\}
	\ \  \text{of the  function}  \ \
	\entropy^{(2)}_{\uhat}(\alphahat) \coloneqq \frac{\UNIT^\T \P^{3}(\alphahat)\UNIT}{3}-\uhat^\T\alphahat \\
	& 
	\text{with gradient}\quad
	\nabla_{\alphahat}	\entropy^{(2)}_{\uhat}(\alphahat) 
	=
	\P(\alphahat)\alphahat - \uhat
	\quad
	\text{and Hessian}
	\quad
	\nabla^2_{\alphahat}	\entropy^{(2)}_{\uhat}(\alphahat) 
	=
	2\P(\alphahat). 
\end{align*}
Hence, this function is convex on the open set~$\AdmissibleSet^+$ and has a unique minimum that satisfies the nonlinear system~$\P(\alphahat)\alphahat =\uhat$. 
These results are partially extended to general~$n$-roots in the following proposition.

\begin{proposition}\label{Lemma1}
	Denote the orthonormal eigenvalue decomposition of the stochastic Galerkin matrix by~$
	\P(\uhat)=\VP(\uhat)\DP(\uhat)\VP^\T(\uhat)$ and define the function
	$$
	\entropy^{(n)}_{\uhat}(\alphahat) \coloneqq \frac{\UNIT^\T \P^{n+1}(\alphahat)\UNIT}{n+1}-\uhat^\T\alphahat. 
	$$	
	Then, the gradient reads as
	\begin{equation*}
		\begin{aligned}
			\nabla_{\alphahat} \entropy^{(n)}_{\uhat}(\alphahat)
			&\, =
			\P^n(\alphahat) \UNIT-\uhat
			+
			\frac{ \mathcal{E}_{n}(\alphahat) }{n+1} 
			\qquad\text{with}
			\\
			\mathcal{E}_{n}(\alphahat)
			&\coloneqq
			\sum\limits_{k=0}^{n}
			\alphahat^{\ast (n-k)} 
			\ast
			\alphahat^{\ast k}
			-
			(n+1)
			\alphahat^{\ast n}.
		\end{aligned}
	\end{equation*}

\end{proposition}

\begin{proof}
	Lemma~\ref{Lemma0} implies the Jacobian
	\begin{alignat}{8}
		&\D_{\alphahat}
		\Big[
		\UNIT^\T 
		\P^{n+1}(\alphahat)
		\UNIT
		\Big]
		&&=
		\UNIT^\T
		\VP(\alphahat)
		\D_{\bar{\alpha}} \DP^{n+1}(\bar{\alpha}) \big\vert_{\bar{\alpha} = \alphahat}
		\VP^\T(\alphahat) \UNIT
		&&+
		\UNIT^\T
		\mathcal{E}_{n}^\T(\alphahat) \nonumber \\
		& &&=
		(n+1)	
		\UNIT^\T
		\P^n(\alphahat)
		&&+
		\UNIT^\T
		\mathcal{E}_{n}^\T(\alphahat) \label{Lemma2proof1}
	\end{alignat}
	$$\quad
	\text{with}\quad
	\mathcal{E}_{n}^\T(\alphahat)
	\coloneqq
	\D_{\bar{\alpha}}
	\VP(\bar{\alpha})\big\vert_{\bar{\alpha} = \alphahat}
	\DP^{n+1}(\alphahat)
	\VP^{\T}(\alphahat)
	\UNIT 
	+
	\VP(\alphahat)
	\DP^{n+1}(\alphahat)
	\D_{\bar{\alpha}}
	\VP^{\T}(\bar{\alpha})\big\vert_{\bar{\alpha} = \alphahat} 
	\UNIT. 
	$$
	
	\noindent	
	Due to~$
	\alphahat^{\ast n}
	=
	\P^{n}(\alphahat) \UNIT
	$, 
	$
	\D_{\bar{\alpha}}
	\P(\bar{\alpha})\big\vert_{\bar{\alpha}=\alphahat}
	\alphahat^{\ast n}
	=
	\P\big(\alphahat^{\ast n}  \big)
	$ 
	and the product rule for matrices, we have
	\begin{align}
		\D_{\alphahat} \Big[		\UNIT^\T
		\P^{n+1}(\alphahat)
		\UNIT
		\Big]
		&=
		\UNIT^\T
		\sum\limits_{k=0}^{n}
		\P^k(\alphahat)
		\D_{\bar{\alpha}}
		\P(\bar{\alpha})\big\vert_{\bar{\alpha}=\alphahat}
		\alphahat^{\ast (n-k)} \nonumber \\
		&=
		\UNIT^\T
		\sum\limits_{k=0}^{n}
		\P^k(\alphahat)
		\P \left( \alphahat^{\ast (n-k)} \right).
		\label{Lemma2proof2}
	\end{align}

	\noindent
	By subtracting equation~\eqref{Lemma2proof2}  from~\eqref{Lemma2proof1}, we obtain the expression
	\begin{equation*} 
		\begin{aligned}
			\mathcal{E}_{n}(\alphahat)\UNIT
			&=
			\sum\limits_{k=0}^{n}
			\P \left( \alphahat^{\ast (n-k)} \right)
			\P^k(\alphahat)
			\UNIT
			-
			(n+1)
			\P^n(\alphahat) \UNIT  \\
			&=
			\sum\limits_{k=0}^{n}
			\alphahat^{\ast (n-k)} 
			\ast
			\alphahat^{\ast k}
			-
			(n+1)
			\alphahat^{\ast n}.
		\end{aligned}
	\end{equation*}		
\end{proof}

\noindent
In principle, general roots may be obtained by minimizing the function~$\entropy^{(n)}_{\uhat}(\alphahat)$, since the error~$
\mathcal{E}_{n}(\alphahat)
$ decreases when the number of gPC bases is increased. More precisely, 
the term~$
\mathcal{E}_{n}(\alphahat)
$	
describes projection errors in the repeated Galerkin products that occur, since the Galerkin product is in general not associative as stochastic Galerkin  matrices  generally do not commute~\cite{S4}. 
Hence, the truncation~$K\in\mathbb{N}_0$ must be sufficiently large compared to the moment~$n\in\mathbb{N}$ and an optimization problem occurs, which causes noteworthy computational cost. Furthermore, convexity and hence a unique solution is not necessarily ensured on the whole domain~$\AdmissibleSet^+$. 
In contrast, the associativity of the Galerkin product in the special case of Haar-type  expansions makes the error~$\mathcal{E}_n=0$ vanish for all~$n\in \mathbb{N}$. The  minimum  
\begin{equation}\label{generalRoot}
	\nRoot
	\coloneqq \argmin\limits_{\alphahat \in \AdmissibleSet^+} \Big\{ \entropy^{(n)}_{\uhat}(\alphahat) \Big\}
	\quad  \text{satisfies}  \quad
	\P^n\left( \nRoot \right) \UNIT = \uhat
\end{equation}	
and is unique, since the Hessian~$\nabla^2_{\alphahat}	\entropy^{(n)}_{\uhat}(\alphahat) 
=
n\P^{n-1}(\alphahat)$ is strictly positive definite for all states~$\alphahat\in\AdmissibleSet^+$. 
Therefore, the inverse~$\mathcal{T}_n^{-1}(\uhat)$, which is defined in Lemma~\ref{Lemma2}, 
is the explicit solution to the well-posed optimization problem~\eqref{generalRoot} for Haar-type expansions.

\section{Applications to hyperbolic conservation laws}\label{SectionConservationLaws}
\noindent
This section presents eigenvalue decompositions for hyperbolic systems with Lipschitz continuous flux functions that involve the non-polynomial expressions derived in Theorem~\ref{LemmaHaar} using Haar-type expansions with uniform distribution~$\xi\sim\mathcal{U}(0,1)$. In the sequel, we consider only these Haar-type expansions with uniform distribution. Furthermore, the domain, where the solution is defined, is stated precisely in terms of the convex sets~\eqref{Definitionsbereich}. 

\subsection{Lipschitz continuous flux function}\label{SectionCorollary1}

We consider the example~\cite[Ex.~3.1]{Guelmame2019}, where the scalar conservation law 
\begin{equation}\label{Oleinik}
	\begin{aligned}
		&\partial_t\u(t,x,\xi)
		+
		\partial_x
		\Big(
		\u^2(t,x,\xi) + \big\vert
		\u(t,x,\xi) 
		\big\vert
		\Big)
		=0 \\
		&\text{with initial values}
		\quad
		\u_0(x,\xi) 
		=
		\sign\big( \widehat{x}(x,\xi)  \big)
	\end{aligned}
\end{equation}
and  initial discontinuity at the point~$
\widehat{x}\big(x,\xi(\omega)\big)\coloneqq 
x- \left(\xi(\omega)-\nicefrac{1}{2}\right)
$~for the random variable~$\xi\sim\mathcal{U}(0,1)$ is analyzed. 
According to~\cite{Guelmame2019}, the pointwise entropy solution is
\begin{equation}\label{SolutionCorollary1}
	\u^{\textup{ref}}(t,x,\xi)
	=
	\begin{cases}
		-1 					& \text{if \ \ } \frac{\widehat{x}(x,\xi)}{t} \in (-\infty,-3),	\\
		\frac{1}{2} \Big(\frac{\widehat{x}(x,\xi)}{t}+1\Big)	& \text{if \ \ } \frac{\widehat{x}(x,\xi)}{t} \in [-3,-1),		\\
		0					& \text{if \ \ } \frac{\widehat{x}(x,\xi)}{t} \in [-1,1),		\\
		\frac{1}{2} \Big(\frac{\widehat{x}(x,\xi)}{t}-1\Big)	& \text{if \ \ } \frac{\widehat{x}(x,\xi)}{t} \in [1,3),			\\
		1					& \text{if \ \ } \frac{\widehat{x}(x,\xi)}{t} \in [3,\infty).
	\end{cases}
\end{equation}
We note that there is a constant intermediate  state due to a 
discontinuity in the characteristic speed. 
The following corollary states the intrusive form. 

\begin{corollary}\label{Corollary1}
	Haar-type stochastic Galerkin formulations to the conservation law~\eqref{Oleinik} read as
	\begin{align*}
		&\partial_t \uhat
		+
		\partial_x
		\Big(
		\uhat^{\ast 2}
		+
		\uhat \ast \shat(\uhat)
		\Big)=\veco
		\quad\text{with initial values}\quad
		\uhat_k(0,x)
		=
		\Big\langle \u_0(x,\xi) , \phi_k(\xi) \Big\rangle_{\mathbb{P}}
		\\
		&
		\text{and generalized Jacobian}\quad
		\D_{\uhat} \widehat{f}(\uhat)
		=
		2 \P(\uhat) + \shat(\uhat)
		=
		\VPhaar \Big[ \, 2\DP(\uhat)+ \sign\big\{ \DP(\uhat) \big\} \Big] \VPhaar^\T.
	\end{align*}
	It is defined for all states~$\uhat \in \mathbb{R}^{K+1}$.
	
\end{corollary}

\begin{proof}
	The expressions for the flux function and its Jacobian follow directly from~Theorem~\ref{LemmaHaar}. 
	
\end{proof}

\subsection{Level set equations}
Two-dimensional level set equations in Hamilton-Jacobi form read as
\begin{equation}\label{detLevelSet}
	\partial_t \varphi(t,\x) + v(\x) \big\lVert \nabla_{\x} \varphi(t,\x) \big\rVert_2   = 0
	\ \ \text{for}\ \ 
	\varphi(0,\x)  \in\mathbb{R}
	\ \ \text{and}\ \
	\x=(x_1,x_2)^\T.
\end{equation}
In the sense of viscosity solutions~\cite{Caselles1992,Crandall1983}, the scalar equation~\eqref{detLevelSet} is equivalent to the {hyperbolic system}
\begin{equation}\label{2dLevelSet}
	\begin{aligned}
		&\partial_t \u(t,\x)
		+
		\partial_{x_1}
		f_1\big(\u(t,\x),v(\x)\big)
		+
		\partial_{x_2}
		f_2\big(\u(t,\x),v(\x)\big)
		=0 \\
		& 
		\text{with flux functions}\quad
		f_1(\u,v)
		=
		\begin{pmatrix}
			v \lVert \u \rVert_2 \\ 0
		\end{pmatrix},
		\quad
		f_2(\u,v)
		=
		\begin{pmatrix}
			0 \\
			v \lVert \u \rVert_2
		\end{pmatrix}
	\end{aligned}
\end{equation}
for~$\u\coloneqq (\u_1,\u_2)^\T$ 
and initial conditions~$\u(0,\x)=\nabla_{\x}\varphi_0(\x) \in\mathbb{R}^2$.
\noindent
The two-dimensional system~\eqref{2dLevelSet} is hyperbolic in the sense that for all unit vectors~${\vec{n}=(n_1,n_2)^\T}$ the matrix
\begin{equation}\label{DeterministicJacobian}
	n_1 \D_{\u} f_1(\u,v) 
	+ 
	n_2 \D_{\u} f_2(\u,v)
	=
	\frac{v}{\lVert \u \rVert_2}
	\begin{pmatrix}
		n_1 \u_1 & n_1 \u_2 \\
		n_2 \u_1 & n_2 \u_2
	\end{pmatrix}
\end{equation}
is diagonalizable with real eigenvalues and a complete set of eigenvectors. 
Note that the Jacobian~\eqref{DeterministicJacobian} reduces in the spatially one-dimensional case to~$
\D_{\u} f(\u,v) = v\, \partial_{\u} \vert\u\vert
$. 
Here, the derivative of the norm is understand as the~generalized gradient~\cite{Clarke1990,Frankowska1989, LipschitzFlux2001}. 
The following theorem states the corresponding stochastic Galerkin formulation and the generalized real characteristic speeds. 

\begin{corollary}\label{Corollary2}
	Assume a Haar-type expansion with gPC modes~$\v$  that account for  random velocities~$v(\xi)$.	
	Then, a stochastic Galerkin formulations for the two-dimensional level set equations~\eqref{2dLevelSet} read as
	\begin{align*}
		&\partial_t \widehat{\u}(t,\x)
		+
		\partial_{x_1}
		\widehat{f_1}\Big(\widehat{\u}(t,\x),\v(\x)\Big)
		+
		\partial_{x_2}
		\widehat{f_2}\Big(\widehat{\u}(t,\x),\v(\x)\Big)
		=\veco \\
		&\text{with flux functions}
		\quad
		\widehat{f_1}\big(\widehat{\u},\v\big)
		=
		\begin{pmatrix}
			\v \ast
			\widehat{ \,
				\lVert \u \rVert_2
				\, } \\ \veco
		\end{pmatrix},\quad
		\widehat{f_2}\big(\widehat{\u},\v\big)
		=
		\begin{pmatrix}
			\veco \\
			\v \ast
			\widehat{ \,
				\lVert \u \rVert_2
				\, }
		\end{pmatrix}.
	\end{align*}
	
	\noindent
	The real spectrum\footnote[1]{With a slight abuse of notation the spectrum is stated in terms of diagonal matrices and~$\zeros$ denotes the zero matrix.} is
	$$
	\sigma\bigg\{
	\vec{n} \cdot
	\D_{\uhat} \widehat{f}\big(\widehat{\u},\v\big) 
	\bigg\}
	=
	\Bigg\{\
	\DP(\v)
	\DP\Big(
	\widehat{ \,
		\lVert \u \rVert_2
		\, }
	\Big)^{-1}
	\Big(
	n_1 \DP\big(\uhatI\big)+n_2 \DP\big(\uhatII\big)
	\Big),\ \zeros\
	\Bigg\},
	$$
	
	
	\noindent
	where eigenvalues read as~$
	\DP\big(
	\widehat{ \,
		\lVert \u \rVert_2
		\, }
	\big)
	=
	\sqrt{\DP\big(\uhatI\big)^2 +\DP\big(\uhatII\big)^2}
	$. 
	It is defined for all states~$\uhat\in\mathbb{R}^{2 (K+1)}$  and satisfies 
	in the sense of generalized gradients  the relation
	$$
	\sigma\Big\{
	\vec{n} \cdot
	\D_{\uhat} \widehat{f}\big(\uhat,\v\big)
	\bigg\}
	=
	\Big\{\
	\DP(\v)\,
	\sign\big\{
	\DP(\uhati)
	\big\}, \ \zeros \
	\Big\}
	\quad\text{for}\quad
	\uhatj = \veco
	\quad\text{with}\quad
	j\neq i.
	$$
\end{corollary}

\begin{proof}
	According to Theorem~\ref{LemmaHaar}, it holds~$
	\D_{\uhati} \widehat{ \,
		\lVert \u \rVert_2
		\, }
	=
	\VPhaar\,
	\DP\big(
	\widehat{\, \lVert \u \rVert \,}
	\big)^{-1}
	\DP\big(
	\uhati
	\big)\,
	\VPhaar^\T
	$. 	
	Defining the block diagonal matrix~$
	\widehat{\VPhaar}
	\coloneqq
	\diag\{1,1\}\otimes\VPhaar, 
	$
	we have the Jacobian
	\begin{equation*}
		\J (\widehat{\u},\v)
		=
		\widehat{\VPhaar} \
		\Bigg[
		\diag\{1,1\}\otimes
		\left(
		\DP\big(\v\big)
		\DP\Big(
		\widehat{\, \lVert \u \rVert \,}
		\Big)^{-1}
		\right)
		\Bigg]
		\begin{pmatrix}
			n_1 \DP\big(\widehat{\,\u_1}\big) & n_1 \DP\big(\widehat{\,\u_2}\big) \\
			n_2 \DP\big(\widehat{\,\u_1}\big) & n_2 \DP\big(\widehat{\,\u_2}\big) 
		\end{pmatrix}
		\
		\widehat{\VPhaar}^\T.
	\end{equation*}
	Due to the sparse   structure, the real spectrum is obtained. In particular for~$\uhatj = \veco$, we have
	\begin{align*}
		\DP(\v)
		\DP\Big(
		\widehat{ \,
			\lVert \u \rVert_2
			\, }
		\Big)^{-1}
		\Big(
		n_1 \DP\big(\uhatI\big)+n_2 \DP\big(\uhatII\big)
		\Big)
		&=
		n_i \DP(\v) \big\vert \DP(\uhati) \big\vert^{-1} \DP(\uhati) \\
		&=
		n_i\DP(\v)  \sign\big\{  \DP(\uhati)  \big\}.
	\end{align*} 	
\end{proof}

\subsection{Gas flow with Lipschitz continuous pressure law}\label{Section43}

We consider a model from~\cite[Sec.~5]{LipschitzFlux2001} related to the $p$-system. The unknowns~$\u\coloneqq (u,v)^\T$ are~the velocity~$u(t,x)$ and the specific volume~$v(t,x)$ of a fluid.  A Lipschitz continuous pressure is stated in terms of the specific volume that  satisfies~$
p_{-}(v^*)<p_{+}(v^*) 
$ 
and~$
p'(v)<0
$,~$
p''(v)>0
$~for~$v\neq v^*$. 
The hyperbolic system reads as $\partial_t \u + \partial_x f(\u)=0$ with flux function
\begin{equation}\label{LeFloch}
	f(\u)
	=
	\begin{pmatrix}
		p(v) \\- u
	\end{pmatrix}
	\quad \text{and pressure}\quad
	p(v)
	\coloneqq
	\begin{cases}
		v^{-\gamma_{1}} & \text{if \  } v<v_{*}, \\
		v^{-\gamma_{2}} 
		+\Delta v_{*}
		& \text{if \  } v>v_{*}, \\
	\end{cases}
\end{equation} 
where the value~$\Delta v_{*}\coloneqq v_{*}^{\gamma_1}-v_{*}^{\gamma_2}$ is chosen such that the pressure law is Lipschitz continuous. 
The following theorem states the stochastic Galerkin formulation that accounts for a Lipschitz continuous random pressure law. 

\begin{corollary}\label{Corollary3}
	Assume a Haar-type expansion with gPC modes~$\vhatStar$ and $\DvhatStar$ that accounts for a random state~$v_{*}(\xi)$. Then, a stochastic Galerkin formulation	
	to the system~\eqref{LeFloch}, describing the gPC modes~$\uhat\coloneqq(\widehat{u},\widehat{v})^\T$ for~$\widehat{u}\in\mathbb{R}^{K+1}$ and~$\widehat{v}\in\mathbb{H}^+$, reads as
	\begin{align*}
		&\partial_t \widehat{\u}(t,\x)
		+
		\partial_{x}
		\widehat{f}\big(\widehat{\u}(t,\x)\big)
		=\veco 
		\quad\text{with flux function}\quad
		\widehat{f}(\uhat) =
		\begin{pmatrix}
			\widehat{p}(\widehat{v}) \\
			\widehat{u}
		\end{pmatrix}\\
		&\begin{aligned}
			\text{and pressure law}\quad
			\widehat{p}(\widehat{v})
			=
			&-\frac{1}{2} 
			\Big[ \shat\big( \widehat{v}-\vhatStar \big) - \UNIT\Big] \ast \widehat{\,v^{-\gamma_1}} \\
			&+
			\frac{1}{2} 
			\Big[ \shat\big( \widehat{v}-\vhatStar \big) + \UNIT\Big] \ast \Big(\widehat{\,v^{-\gamma_2}}
			+\DvhatStar\Big).
		\end{aligned}
	\end{align*}	
	Furthermore, the generalized spectrum is
	$
	\sigma\big\{
	\D_{\uhat} \widehat{f}(\uhat)
	\big\}
	=
	\big\{\pm
	\DP_{p'}^{\nicefrac{1}{2}}(\widehat{v})
	\big\}
	$ with 
	$$
	\DP_{p'}^{\nicefrac{1}{2}}(\widehat{v})
	\coloneqq
	\frac{\sqrt{\gamma_1}}{2}
	\bigg[
	\DP\Big( \shat\big( \widehat{v}-\vhatStar\big)\Big) - \indikator\Big) \bigg]
	\DP^{-\frac{\gamma_1+1}{2}} (\widehat{v}) 
	-\frac{\sqrt{\gamma_2}}{2}  
	\bigg[
	\DP\Big( \shat\big( \widehat{v}-\vhatStar\big)\Big) + \indikator\bigg]  \DP^{-\frac{\gamma_2+1}{2}}(\widehat{v}).
	$$
	
\end{corollary}

\begin{proof}
	We write the Lipschitz continuous pressure~\eqref{LeFloch} as
	$$
	p(v)=
	-
	\frac{1}{2}
	\Big[
	\sign(v-v^*) - 1 
	\Big] v^{-\gamma_1}
	+
	\frac{1}{2}
	\Big[
	\sign(v-v^*) + 1 
	\Big]\Big( v^{-\gamma_2}+\Delta v_{*} \Big).
	$$	
	Theorem~\ref{LemmaHaar} yields the projected pressure law~$\widehat{p}(\widehat{v})$. 
	Furthermore, Theorem~\ref{LemmaHaar} implies the generalized Jacobian 
	$$
	\D_{\uhat}
	\widehat{f}(\uhat)
	=
	\begin{pmatrix}
		\zeros & \D_{\widehat{v}} \widehat{p}(\widehat{v}) \\
		\indikator & \zeros
	\end{pmatrix}
	\quad\text{with}\quad
	\begin{aligned}
		\D_{\widehat{v}} \widehat{p}(\widehat{v})
		&=
		\frac{\gamma_1}{2} 
		\VPhaar 
		\DP\Big( \shat\big( \widehat{v}-\vhatStar \big) - \UNIT\Big)
		\DP^{-(\gamma_1+1)} (\widehat{v})\VPhaar^\T \\
		&-
		\frac{\gamma_2}{2} 
		\VPhaar 
		\DP\Big( \shat\big( \widehat{v}-\vhatStar \big) + \UNIT\Big)  \DP^{-(\gamma_2+1)}(\widehat{v}) \VPhaar^\T.
	\end{aligned}
	$$
	Hence, we have the eigenvalue decomposition
	\begin{align*}
		&\D_{\uhat}
		\widehat{f}(\uhat)
		=
		\widehat{T}(\widehat{v})\,
		\diag\Big\{  
		-\DP_{p'}^{\nicefrac{1}{2}}(\widehat{v})  ,\,
		\DP_{p'}^{\nicefrac{1}{2}}(\widehat{v}) 
		\Big\}
		\, \widehat{T}(\widehat{v})^{-1}\\
		&\text{for}\quad
		\widehat{T}(\widehat{v})
		\coloneqq
		\begin{pmatrix}
			\VPhaar & \\ & \VPhaar
		\end{pmatrix}
		\begin{pmatrix}
			- \DP_{p'}^{\nicefrac{1}{2}}(\widehat{v}) 
			& \DP_{p'}^{\nicefrac{1}{2}}(\widehat{v}) \\
			\indikator & \indikator
		\end{pmatrix}.
	\end{align*}
\end{proof}

\subsection{Isentropic Euler equations}

Following~\cite[Sec.~18.3]{Leveque}, we consider two-dimensional isentropic Euler equations that describe the density~$\rho$ of a gas and the mass flux~$q_i$ with respect to the $x_i$-direction.  The flux functions, describing the temporal propagation of the unknown~$\u=(\rho,q_1,q_2)^\T$, read as
\begin{equation}\label{EulerDeterministic}
	f_1(\u)
	=
	\begin{pmatrix}
		q_1 \\ \frac{q_1^2}{\rho} + p(\rho) \\ \frac{q_1 q_2}{\rho}
	\end{pmatrix}
	\quad \text{and} \quad
	f_2(\u)
	=
	\begin{pmatrix}
		q_2 \\ \frac{q_1 q_2}{\rho} \\ \frac{q_2^2}{\rho} + p(\rho) 
	\end{pmatrix}
\end{equation}
with nonlinear pressure law~$
\rho^\gamma>0$. 
It has been proposed in~\cite{S5} to use \textbf{Roe variables}~$\RoeDetI\coloneqq \sqrt{\rho}$ and $\RoeDetII_i(\u) = \nicefrac{q_i}{\RoeDetI}$ as auxiliary variables that yield the stochastic Galerkin formulations
\begin{align}
	&\widehat{f_1}(\uhat)
	=
	\begin{pmatrix}
		\widehat{q_1} \\
		\RoeII_1(\uhat) \ast \RoeII_1(\uhat)
		+
		\widehat{p}(\widehat{\rho}) \\
		\RoeII_1(\uhat) \ast \RoeII_2(\uhat)
	\end{pmatrix}
	\quad\text{and}\quad
	\widehat{f_2}(\uhat)
	=
	\begin{pmatrix}
		\widehat{q_2} \\
		\RoeII_1(\uhat) \ast \RoeII_2(\uhat) \\
		\RoeII_2(\uhat) \ast \RoeII_2(\uhat)
		+
		\widehat{p}(\widehat{\rho}) 
	\end{pmatrix} \nonumber \\
	&\text{for}\quad
	\widehat{\sqrt{\rho}\, }(\widehat{\rho})
	\coloneqq
	\argmin\limits_{\hat{\alpha} \in \AdmissibleSet^+} \Big\{ \entropy^{(2)}_{\hat{\rho}}(\hat{\alpha}) \Big\}
	\quad  \text{and}  \quad
	\RoeII_i(\uhat) 
	\coloneqq
	\P\Big(\widehat{\sqrt{\rho}\, }(\widehat{\rho}) \Big)
	\widehat{q_i}. \label{OptimizationProblemEuler}
\end{align}
Furthermore, this choice leads to a hyperbolic formulation of  one-dimensional isothermal Euler equations for any~gPC expansions~\cite{FettesPaper}. 
Here, we consider the Haar-type formulations that allow for an extension to non-polynomial pressure laws and multiple space dimensions without optimization problems.

\begin{corollary}\label{corollary4}
	Define the  variables~$
	\velocityIII(\uhat)
	\coloneqq 
	\VPhaar
	\DP_{\velocityIII}(\uhat)
	\VPhaar^\T \UNIT
	$, 
	$
	\DP_{\velocityIII}(\uhat)
	\coloneqq
	\DP(\widehat{q_i}) \DP(\widehat{\rho})^{-1}
	$ and the matrix 
	$\vec{n}\cdot\DP_{\velocity}(\uhat)
	\coloneqq 
	n_1 \DP_{\velocityI}(\uhat)
	+
	n_2 \DP_{\velocityII}(\uhat)
	$. 
	Then, 
	the projected pressure law is 
	$\widehat{p}(\widehat{\rho})
	=
	\VPhaar\DP(\widehat{\rho})^\gamma \VPhaar^\T \UNIT$ 
	and 
	a Haar-type stochastic Galerkin formulation, which is defined for all states~$\uhat\coloneqq (\widehat{\rho},\widehat{q})^\T$ satisfying $\widehat{\rho}\in\mathbb{H}^+ $, to the two-dimensional isentropic Euler equations~\eqref{EulerDeterministic} reads as
	\noindent
	\begin{align*}
		&\partial_t \widehat{\u}(t,\x)
		+
		\partial_{x_1}
		\widehat{f_1}\Big(\widehat{\u}(t,\x)\Big)
		+
		\partial_{x_2}
		\widehat{f_2}\Big(\widehat{\u}(t,\x)\Big)
		=\veco \quad \text{with flux functions} \\
		& 
		\widehat{f_1}\big(\widehat{\u}\big)
		=
		\begin{pmatrix}
			\velocityI(\uhat) \ast \widehat{\rho} \\
			\velocityI(\uhat)^{\ast 2} \ast \widehat{\rho}  + \widehat{p}(\widehat{\rho})\\
			\velocityI(\uhat)\ast\velocityII(\uhat) \ast \widehat{\rho}  \\	\end{pmatrix},\quad
		\widehat{f_2}\big(\widehat{\u}\big)
		=
		\begin{pmatrix}
			\velocityII(\uhat) \ast \widehat{\rho} \\
			\velocityI(\uhat)\ast\velocityII(\uhat) \ast \widehat{\rho}  \\
			\velocityII(\uhat)^{\ast 2} \ast \widehat{\rho}  + \widehat{p}(\widehat{\rho})\\	
		\end{pmatrix} \\
		&\text{and real spectrum}\quad
		\sigma\Big\{
		\vec{n} \cdot
		\D_{\uhat} \widehat{f}\big(\widehat{\u},\v\big) 
		\Big\}
		=
		\Big\{\
		\vec{n}\cdot\DP_{\velocity}(\uhat)
		\pm
		\sqrt{\gamma}\, \DP^{\frac{\gamma-1}{2}}( \widehat{\rho} ), \
		\vec{n}\cdot\DP_{\velocity}(\uhat) \ 
		\Big\}.
	\end{align*} 
\end{corollary}

\begin{proof}
	The solution to the optimization problem is obtained by Theorem~\ref{LemmaHaar}, i.e.
	$$
	\RoeII_i(\uhat) 
	=
	\VPhaar \DP( \widehat{q_i}) \DP(\widehat{\rho})^{-\nicefrac{1}{2}}  
	\VPhaar^\T \UNIT
	\quad\text{and}\quad
	\RoeII_i(\uhat) 
	\ast 
	\RoeII_j(\uhat) 
	=
	\VPhaar
	\DP( \widehat{q_i})
	\DP( \widehat{q_j})
	\DP( \widehat{\rho})^{-1}
	\VPhaar^\T 
	\UNIT.
	$$
	Lemma~\ref{Lemma0} yields~$\D_{\widehat{\rho}}
	\big[
	\DP( \widehat{\rho})
	\VP^\T 
	\UNIT
	\big] = \indikator  $. Hence, we have the Jacobian
	$$
	\D_{\widehat{\rho}}
	\Big[
	\RoeII_i(\uhat) 
	\ast 
	\RoeII_j(\uhat) 
	\Big]
	=
	\VPhaar
	\DP( \widehat{q_i})
	\DP( \widehat{q_j})
	\D_{\widehat{\rho}}
	\Big[
	\DP( \widehat{\rho})^{-1}
	\VPhaar^\T 
	\UNIT
	\Big]
	=
	-
	\VPhaar
	\DP_{\velocityI}(\uhat)
	\DP_{\velocityII}(\uhat)
	\VPhaar^\T.
	$$
	Due to Theorem~\ref{LemmaHaar}, we have the expression~$
	\D_{\hat{\rho}}
	\widehat{p}(\widehat{\rho})
	=
	\VPhaar 
	\big[ \gamma \DP(\widehat{\rho})^{\gamma-1}\big]
	\VPhaar^\T
	$ and the Jacobian of the flux function reads as
	\begin{align*}
		&\vec{n} \cdot \D_{\uhat} \widehat{f}(\uhat) 
		=
		n_1 \D_{\uhat} \widehat{f_{1}}(\uhat) 
		+
		n_2 \D_{\uhat} \widehat{f_{2}}(\uhat) 
		\quad\text{for}\quad
		\widehat{\VPhaar} \coloneqq \diag\big\{ \VPhaar,\VPhaar,\VPhaar\big\}
		\quad\text{and}\\
		&\D_{\uhat} \widehat{f_{1}}(\uhat) 
		=
		\widehat{\VPhaar} 
		\begin{pmatrix}
			\zeros & \indikator & \zeros \\
			\gamma \DP(\widehat{\rho})^{\gamma-1} - \DP_{\velocityI}(\uhat)^2 & 2\DP_{\velocityI}(\uhat) & \zeros \\
			-\DP_{\velocityI}(\uhat)\DP_{\velocityII}(\uhat) & \DP_{\velocityII}(\uhat) & \DP_{\velocityI}(\uhat)
		\end{pmatrix} \widehat{\VPhaar}^\T, \\
		&\D_{\uhat} \widehat{f_{2}}(\uhat) 
		=
		\widehat{\VPhaar} 
		\begin{pmatrix}
			\zeros & \zeros & \indikator \\
			-\DP_{\velocityI}(\uhat)\DP_{\velocityII}(\uhat) & \DP_{\velocityII}(\uhat) & \DP_{\velocityI}(\uhat) \\
			\gamma \DP(\widehat{\rho})^{\gamma-1} -\DP_{\velocityII}(\uhat)^2 & \zeros & 2\DP_{\velocityII}(\uhat) \\
		\end{pmatrix} 
		\widehat{\VPhaar}^\T.
	\end{align*}
	
	\noindent
	Due to the sparse block diagonal structure, the real spectrum is obtained. 
	
\end{proof}

\noindent

\noindent
Finally, we remark that all involved matrices can be exactly precomputed by using an appropriate quadrature rule with a finite number of   nodes~$\xi^{(1)},\ldots,\xi^{(Q)}$. 
According to~\cite[Sec.~5.1]{FettesPaper}, the positivity of the gPC expansions $\Pi_K\big[\widehat{\rho}\big]\big(\xi^{(q)}\big)>0$ at  all points~$q=1,\dots,Q$ implies~$\widehat{\rho}\in\mathbb{H}^+$, which guarantees hyperbolicity of the stochastic Galerkin formulation in Theorem~\ref{corollary4}.  Hence, the
connection between the hyperbolicity of the Galerkin and the original system is established.

\section{Numerical results}\label{Numerics}
\noindent
All numerical experiments are performed with a globally third order scheme applied on the deterministic Galerkin system. We construct the numerical scheme applying the method of lines and the local Lax-Friedrichs flux. The scheme employs the classical third order strong stability preserving (SSP) Runge-Kutta (RK) method with three stages~\cite{Jiang1996} for the time discretization and third-order CWENO reconstructions for the high-order spatial discretization. In particular, for the one-dimensional problems we consider the CWENO method of~\cite{Visconti2018}, whereas for the two-dimensional problems we use the truly 2D reconstruction described in~\cite{SempliceCocoRusso,CastroSemplice2019} which avoids dimensional splitting. 
The use of CWENO reconstructions is also suited for the high-order numerical treatment of balance laws where the source terms are integrated with a Gaussian quadrature formula matching the order of the scheme. The scheme is implemented in the finite volume solver \emph{SteFVi}~\cite{gerster_stephan_2022_5870686}. The computational grid is set up with uniform cells and corresponding ghost cells taking into account the boundary conditions. The CFL number is always set to $0.45$. 

\begin{remark}
	The flux functions of the derived systems are not smooth and eigenvalues may coincide. Hence, the systems are not in the common strictly hyperbolic setting and  non-classical wave phenomena occur~\cite{Philipp}. Thus, numerical solvers usually based on approximations of solutions to Riemann problems might fail to recover non-classical wave patterns. However, it has been  shown in~\cite{muller2006riemann} that most  solvers are able to correctly approximate solutions to non-classical hyperbolic systems of PDEs provided that the time step is small enough.
	
	The subdifferential of the absolute value at zero is the whole interval~$[-1,1]$ and  the generalized spectra derived in Section~\ref{SectionConservationLaws}  enter the local Lax-Friedrichs flux as the maximum of the characteristic speeds. This ensures sufficient numerical viscosity and, as our computations show, the SSPRK-CWENO schemes approximate  the non-classical waves correctly.

	We recall that the positive definiteness assumption~\eqref{Definitionsbereich} holds if all realizations are positive. In the following numerical experiments for the gas dynamics described in Corollary~\ref{Corollary3} and~\ref{corollary4}, initial values are chosen away from vacuum states such that the discretization does not violate this property. 		
\end{remark}

\subsection{Lipschitz continuous flux function}
We illustrate the theoretical results that are based on Haar-type expansions by means of the random initial value problem~\eqref{Oleinik} in Subsection~\ref{SectionCorollary1}. 
In particular, we consider the wavelet systems~$\mathbb{W}\big[\VPhaar_{\cos}\big]$ and~$\mathbb{W}\big[\VPhaar_J\big]$ that are generated by the discrete cosine and Haar transform. Initial values for the stochastic Galerkin formulation presented in Corollary~\ref{Corollary1} are obtained by the orthogonal projection of the random, space-dependent signum function~$ \sign\left(\widehat{x}\big(x,\xi(\omega)\big)\right)$. The first mode, i.e.~the mean~$
\left\langle \sign\big(\widehat{x}(x,\xi)\big) , \phi_0(\xi) \right\rangle_{\mathbb{P}}
=
\mathbb{E}\left[\sign\big(\widehat{x}(x,\xi)\big)\right]
$, yields the scaling function that is illustrated in Figure~\ref{Example1_Fig1} as black line with the scale shown at the left axis. The detail functions are obtained by the projections~$
\left\langle\sign\big(\widehat{x}(x,\xi)\big) , \phi_k(\xi) \right\rangle_{\mathbb{P}}
$ for $k\geq 1$ and shown at the right axis. 
We observe that the details are ``smoother'' if the cosine transform (left panel) is used in comparison to the Haar transform~(right panel).

The obtained scaling and detail functions define the approximation~$\Pi_K\big[\uhat(0,x)\big](\xi) $ by the 
wavelet system~\eqref{WaveletSystem}, which is shown in Figure~\ref{Example1_Fig2}. 
The realizations are stated by the colorbar.  
In comparison to Figure~\ref{Example1_Fig1}, where~the realizations~$\sign\big(\widehat{x}\big(x,\xi(\omega)\big)\big)\in\{-1,1\}$ define two areas, which are separated by a straight line given by the scaling function, a piecewise approximation is observed in Figure~\ref{Example1_Fig2}, which results from a compression of input data. 

\begin{figure}[H]
	\begin{minipage}{0.49\textwidth}
		\begin{center}	
			\textbf{cosine transform}
			
			\scalebox{1}{\includegraphics[width=\linewidth]{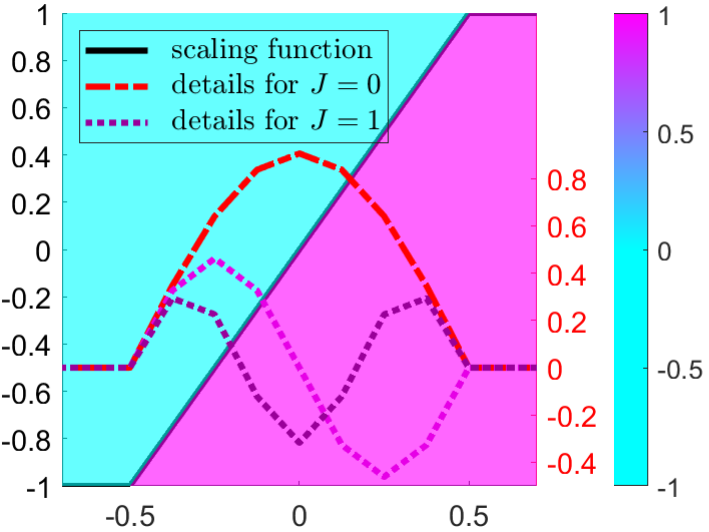}}
			
		\end{center}	
	\end{minipage}
	\hfil
	\begin{minipage}{0.49\textwidth}
		\begin{center}	
			\textbf{Haar transform}
			
			\scalebox{1}{\includegraphics[width=\linewidth]{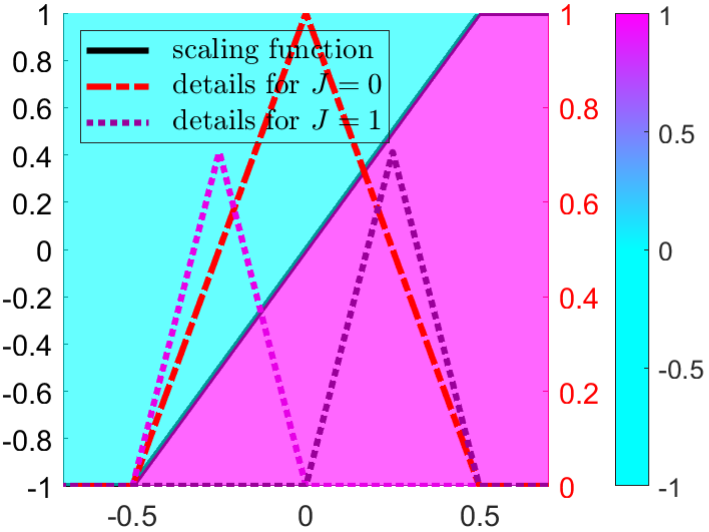}}
			
		\end{center}	
	\end{minipage}
	
	\caption{Projection of random initial values~$
		\left\langle\sign\big(\widehat{x}(x,\xi)\big) , \phi_k(\xi) \right\rangle_{\mathbb{P}}
		$. The left panel accounts for the cosine transform~\eqref{CosineTrafo}, the right panel for the Haar transform~\eqref{HaarTrafo}. The colorbar states realizations~$\sign\big(\widehat{x}\big(x,\xi(\omega)\big)\big)\in\{-1,1\}$. }
	\label{Example1_Fig1}
\end{figure}

\begin{figure}[H]
	\begin{minipage}{0.49\textwidth}
		\begin{center}	
			\textbf{cosine transform}
			
			\scalebox{1}{\includegraphics[width=\linewidth]{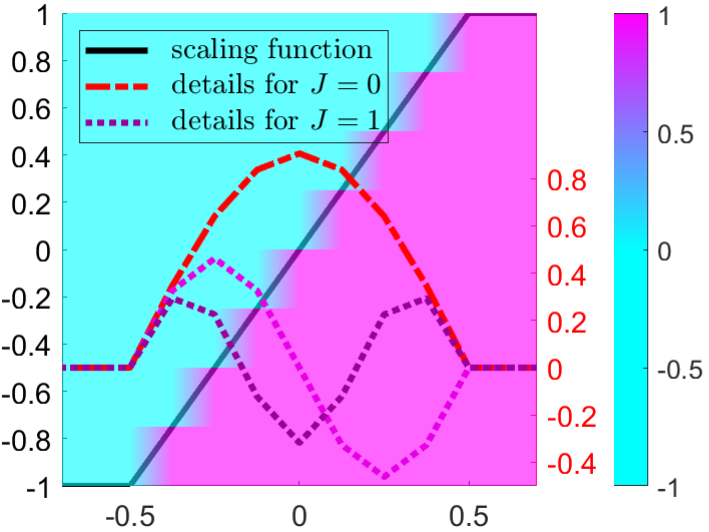}}
			
		\end{center}	
	\end{minipage}
	\hfil
	\begin{minipage}{0.49\textwidth}
		\begin{center}	
			\textbf{Haar transform}
			
			\scalebox{1}{\includegraphics[width=\linewidth]{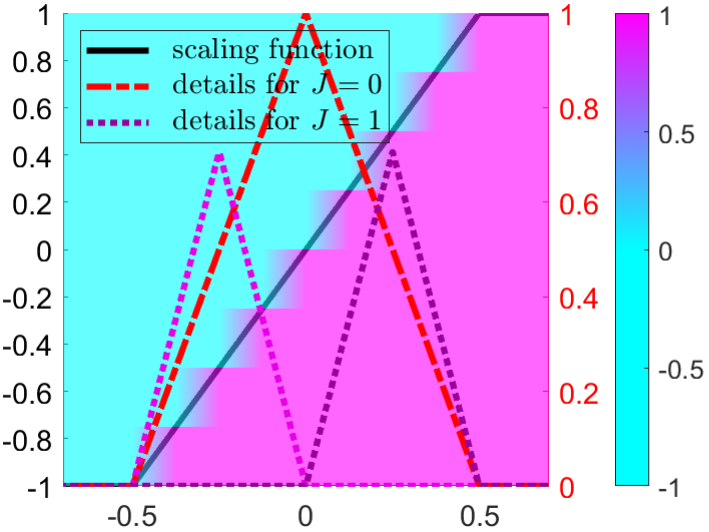}}
			
		\end{center}	
	\end{minipage}
	
	\caption{Realizations~$\Pi_K\big[\uhat(0,x)\big]\big(\xi(\omega)\big) $ that result from the cosine transform~(left) and Haar transform (right) applied to the random initial values~$\sign\big(\widehat{x}(x,\xi)\big)$.}
	\label{Example1_Fig2}
\end{figure}

\begin{figure}[H]
	\begin{minipage}{0.49\textwidth}
		\begin{center}	
			\textbf{cosine transform}
			
			\scalebox{1}{\includegraphics[width=\linewidth]{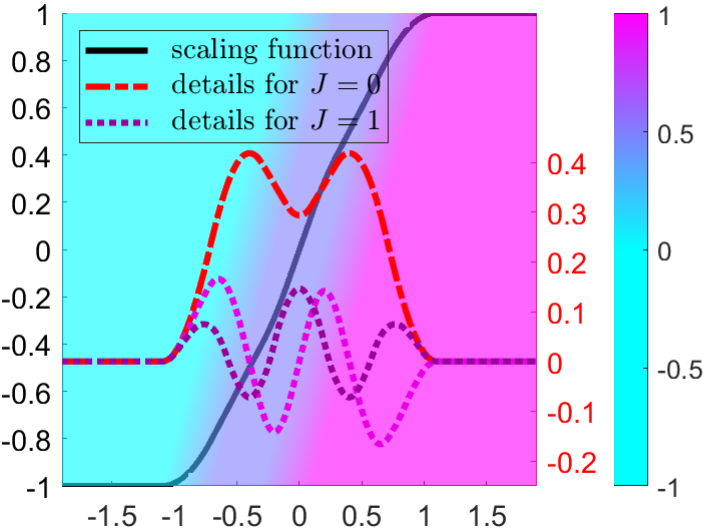}}
			
		\end{center}	
	\end{minipage}
	\hfil
	\begin{minipage}{0.49\textwidth}
		\begin{center}	
			\textbf{Haar transform}
			
			\scalebox{1}{\includegraphics[width=\linewidth]{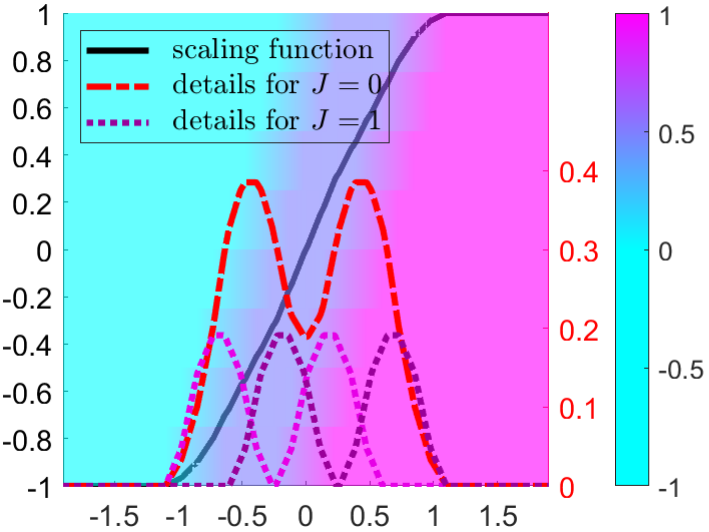}}
			
		\end{center}	
	\end{minipage}
	
	
	\vspace{6mm}
	
	\begin{minipage}{0.49\textwidth}
		\begin{center}	
			\textbf{solution in }$\boldsymbol{t=0.2}$			
			\scalebox{1}{\includegraphics[width=\linewidth]{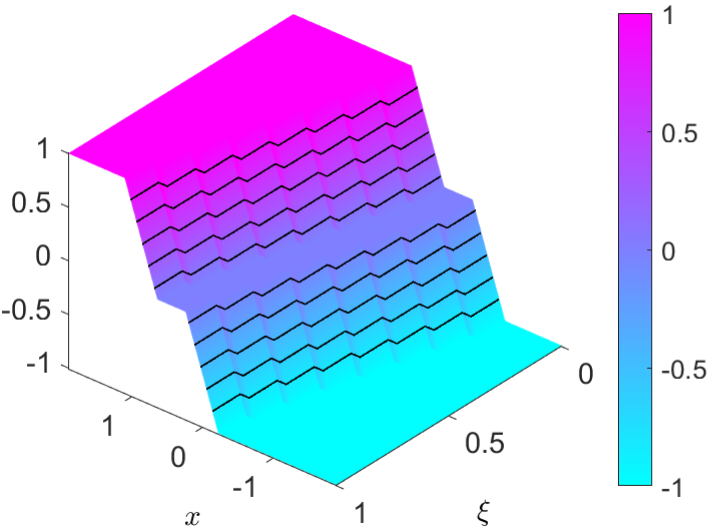}}
			
		\end{center}	
	\end{minipage}
	\hfil
	\begin{minipage}{0.49\textwidth}
		\begin{center}	
			\textbf{mean squared error}
			
			\scalebox{1}{\includegraphics[width=\linewidth]{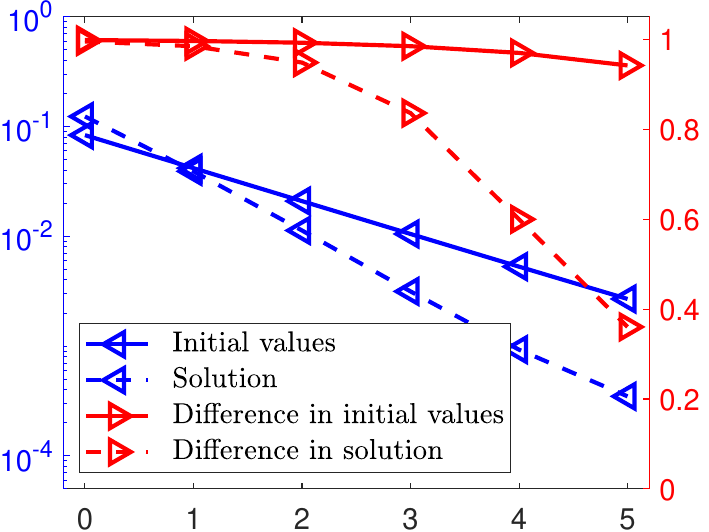}}
			
		\end{center}	
	\end{minipage}
	
	\caption{Upper panels show the realizations~$\Pi_K\big[\uhat(0.2,x)\big]\big(\xi(\omega)\big) $ at time~${t=0.2}$ that result from the cosine transform~(left) and the Haar transform (right) applied to the random initial values~$\sign\big(\widehat{x}(x,\xi)\big)$.  
		The left,  lower panel shows the obtained solution at time $t=0.2$ for the cosine transform. The right, lower panel shows the MSEs for initial values blue, solid line and the MSEs at time~${t=0.2}$ (blue, dashed line)   with respect to  the left axis for the levels $J=0,\ldots,5$. The ratios~$\textup{MSE}_{\textup{Haar} } \textup{MSE}_{\textup{cos}}^{-1}$ of the MSEs~\eqref{MSE}  between the  Haar and cosine transform are shown at the right axis.}
	\label{Example1_Fig3}
\end{figure}

The gPC modes~$\uhat(t,x)$ are propagated over time by the hyperbolic system, derived in Corollary~\ref{Corollary1}. The solution in turn defines again an approximation~$\Pi_K\big[\uhat(t,x)\big](\xi)$, which is illustrated in the upper panels of Figure~\ref{Example1_Fig3} for both the cosine and Haar transform. 
In particular, the approximated solution at time~$t=0.2$ is shown in the left, lower panel for the cosine transform with level~$J=2$. 
The mean squared error
\begin{equation}\label{MSE}
	\textup{MSE}\Big[
	\uhat,\u^{\textup{ref}}
	\Big](t)
	\coloneqq
	\int
	\mathbb{E}\bigg[
	\Big(
	\Pi_K\big[\uhat\big](t,x,\xi)
	-
	\u^{\textup{ref}}(t,x,\xi)
	\Big)^2
	\bigg]
	\d x
\end{equation}
between the approximation~$\Pi_K\big[\uhat(0.2,x)\big](\xi)$ and the reference solution~\eqref{SolutionCorollary1} 
for the Haar transform  is shown in the lower, right panel with respect to the left axis. Furthermore, the difference of the mean squared error of the solution that is obtained by the Haar and cosine transform is stated in terms of the ratio~$\textup{MSE}_{\textup{Haar} } \textup{MSE}_{\textup{cos}}^{-1}$, i.e.~the ratio between the mean squared errors~\eqref{MSE} that are obtained by the Haar and cosine transform. 
As expected, we observe a decreasing MSE for an increasing number of basis functions. In this particular example, the Haar expansion results in a lower MSE than the cosine transform.

\subsection{Level set equations}
We consider the two-dimensional level set equations~\eqref{2dLevelSet} with uncertain velocity~$v\sim\mathcal{U}[\nicefrac{1}{2},1]$. Initial values are deterministic and read as
$$\u(0,\x) 
=\begin{cases}
	\,\ \ (1,0)^\T & \text{if \ \ }
	\x\in[-2,2]^2, \\
	(-1,0)^\T & \text{else.} 
\end{cases}
$$

\begin{figure}[H]
	\begin{minipage}{0.49\textwidth}
		\begin{center}	
			\textbf{mean for intrusive formulation}
			
			\scalebox{1}{\includegraphics[width=\linewidth]{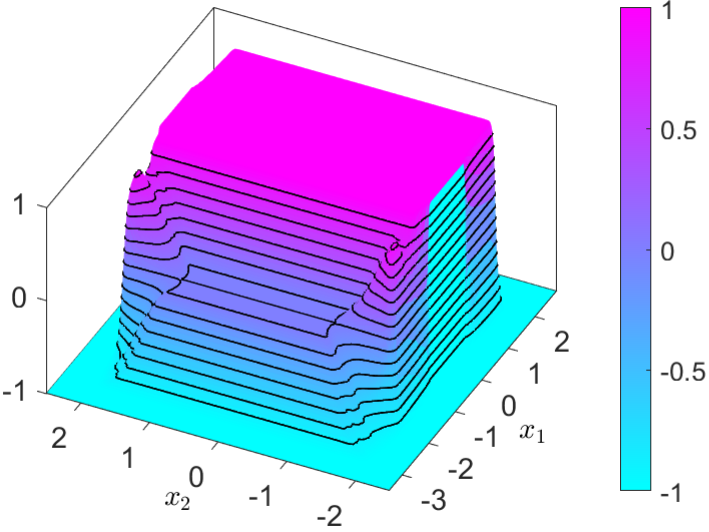}}
			
		\end{center}	
	\end{minipage}
	\hfil
	\begin{minipage}{0.49\textwidth}
		\begin{center}	
			\textbf{standard deviation}
			
			\scalebox{1}{\includegraphics[width=\linewidth]{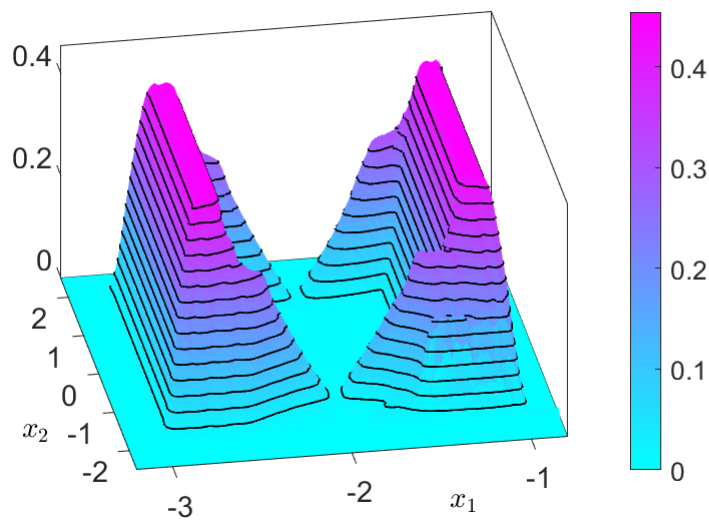}}
			
		\end{center}	
	\end{minipage}
	
	\caption{Left: Mean~$(\widehat{\,\textup{\textbf{u}}_1})_0(1,\x)$ at time~$t=1$ to the intrusive formulation in Corollary~\ref{Corollary2}. Right: Standard deviation given by~$
		\big\lVert \big(
		(\widehat{\,\textup{\textbf{u}}_1})_1, 
		\ldots,
		(\widehat{\,\textup{\textbf{u}}_1})_K
		\big)\big\rVert_2(1,\x)$.
	}
	\label{Example2_Fig1}
\end{figure}

The left panel of Figure~\ref{Example2_Fig1} 
shows the mean for the unknown~$\widehat{\,\textup{\textbf{u}}_1}$ 
and the right panel the standard deviation 
that are obtained by the intrusive formulation in Corollary~\ref{Corollary2}.

	\begin{remark}
		
		According to~\cite[Lem.~3.1]{GersterHertyCicip2020}, which is generalized in Proposition~\ref{Lemma1},  
		the  Euclidean norm~$\big\lVert \u(\xi) \big\rVert$ of a $d$-dimensional random variable~$\u\big(\xi(\omega)\big)\in\mathbb{R}^d$ can be approximately represented for arbitrary basis functions in terms of the minimization problem
		\begin{equation*}
			\widehat{ \,
				\lVert \u \rVert
				\, }
			=
			\argmin\limits_{\alphahat \in \AdmissibleSet^+} \Big\{ \entropy^{(2)}_{	\sum\limits_{i=1}^d
				\uIII^{\ast 2}
			}(\alphahat) \Big\}
			=
			\argmin\limits_{\alphahat \in \AdmissibleSet^+}
			\Bigg\{
			\frac{\hat{\alpha}^\T \P(\hat{\alpha}) \hat{\alpha} }{3}
			\ - \
			\hat{\alpha}^\T \,
			\sum\limits_{i=1}^d
			\uIII^{\ast 2}
			\Bigg\}. 
		\end{equation*}
		This motivates a stochastic Galerkin formulation for general basis functions that is of the form
		\begin{equation}
			\begin{aligned}
				&\partial_t \widehat{\u}(t,x)
				+
				\partial_{x}
				\widehat{f}\Big(\widehat{\u}(t,x),\v(t,x)\Big)
				=0
				\quad
				\text{with}
				\quad
				\widehat{f}\big(\widehat{\u},\v\big)
				=
				\v \ast
				\widehat{ \,
					\lVert \u \rVert
					\, } \\
				& \text{and Jacobian}\quad
				\D_{\widehat{\u}}
				\widehat{f}\big(\widehat{\u},\v\big)
				=
				\P\big(\v\big)
				\P\left(
				\widehat{ \,
					\lVert \u \rVert
					\, }
				\right)^{-1}
				\P\big( \widehat{\u} \big)
				\label{LevelSetNonHyperbolic}
			\end{aligned}
		\end{equation}
		for constant velocities 
		in one spatial dimension. 
		Indeed, the system~\eqref{LevelSetNonHyperbolic} is equivalent to the formulation presented in Corollary~2 provided that Haar-type expansions are used, as explained in Section~\ref{SectionChallenges}. However, it has been shown in~\cite{LevelSet} that for general basis function the system~\eqref{LevelSetNonHyperbolic} is  \emph{not} hyperbolic.
		This loss of hyperbolicity can be circumvented by considering the capacity form
		\begin{equation*}
			\begin{aligned}
				&\partial_t \Big[\P\big(\v(t,x)\big) \widehat{\u}(t,x) \Big]
				+
				\partial_{x}
				\widetilde{f}\Big(\widehat{\u}(t,x)\Big)
				=0
				\quad
				\text{with flux function}
				\quad
				\widetilde{f}\big(\widehat{\u}\big)
				=
				\widehat{ \,
					\lVert \u \rVert
					\, } \\
				& \text{and Jacobian}\quad
				\D_{\widehat{\u}}
				\widetilde{f}\big(\widehat{\u}\big)
				=
				\P\left(
				\widehat{ \,
					\lVert \u \rVert
					\, }
				\right)^{-1}
				\P\big( \widehat{\u} \big)
			\end{aligned}
		\end{equation*}
		that has real eigenvalues and a complete set of eigenvectors in the case~${
			\widehat{ \,
				\lVert \u \rVert
				\, }\in\mathbb{H}^+
		}$. We refer to~\cite{LevelSet} where a 
		capacity-form differencing scheme  based on a non-uniform effective grid~\cite[Sec.~6.16, Sec.~6.17]{Leveque} has been introduced that preserves hyperbolicity even in the multi-dimensional case and for non-constant velocities. 
		However, the drawback is the minimization problem in  Proposition~\ref{Lemma1} that is ill-posed close to the random zero-level set, which requires  additional regularization techniques~\cite[Sec.~4]{LevelSet}.
		
		Figure~\ref{Example2Comparison} compares the 
		formulation that is presented in Corollary~\ref{Corollary2} with the  hyperbolic stochastic Galerkin formulation from~\cite{LevelSet} using Legendre  polynomials in the one-dimensional case.  
		We consider again deterministic initial values and uncertainties in the velocity, more precisely
		$$\u(0,x) 
		=\begin{cases}
			-1 & \text{if \ \ }
			x<0, \\
			\ \ 1 & \text{if \ \ }
			x>0
		\end{cases}
		\qquad\text{and}\qquad
		v\sim\mathcal{U}\big[\nicefrac{1}{2},\nicefrac{3}{2}\big].
		$$
		Apart from the random zero-level set close to~$x\approx 0$ where it holds~$\u\big(t,x,\xi(\omega)\big) \approx 0$, 
		we observe good agreement between the reference solution (red line) and the stochastic Galerkin formulations (blue dotted). 
		The fact that the optimization problem in Proposition~\ref{Lemma1} is not necessarily well-posed for Legendre polynomials becomes  noticeable close to~$x\in[-2.5,2.5]$. In contrast, the optimization problem is well-posed for Haar-type expansions and hence a good approximation quality is  observable in the right panel.	
		\begin{figure}[H]
			\begin{minipage}{0.49\textwidth}
				\begin{center}	
					\textbf{Legendre polynomials}
					
					\scalebox{1}{\includegraphics[width=\linewidth]{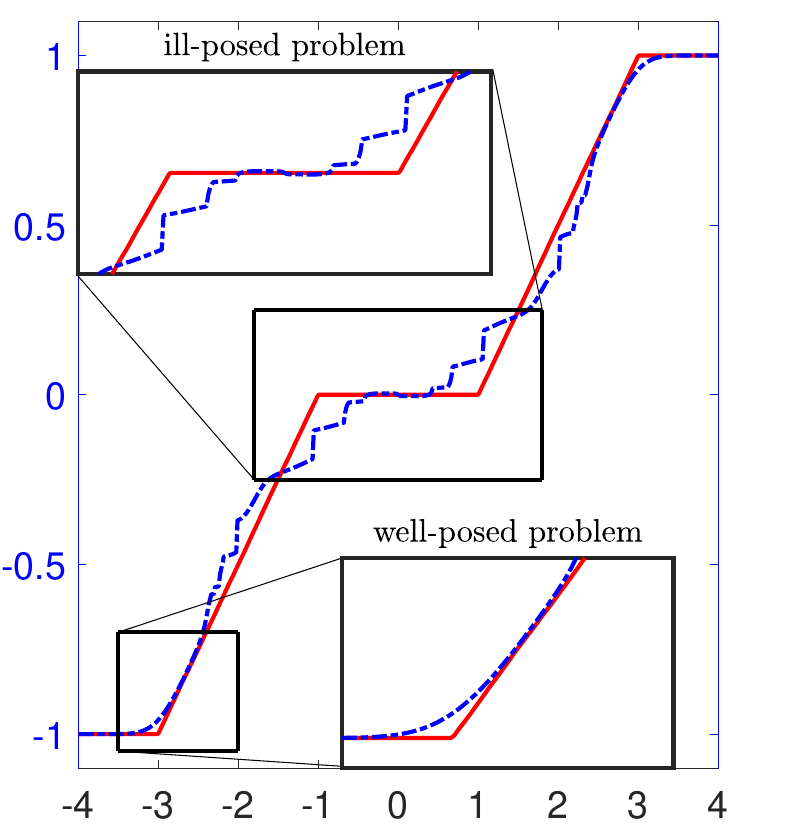}}
					
				\end{center}	
			\end{minipage}
			\hfil
			\begin{minipage}{0.49\textwidth}
				\begin{center}	
					\textbf{formulation in Corollary~\ref{Corollary2}}
					
					\scalebox{1}{\includegraphics[width=\linewidth]{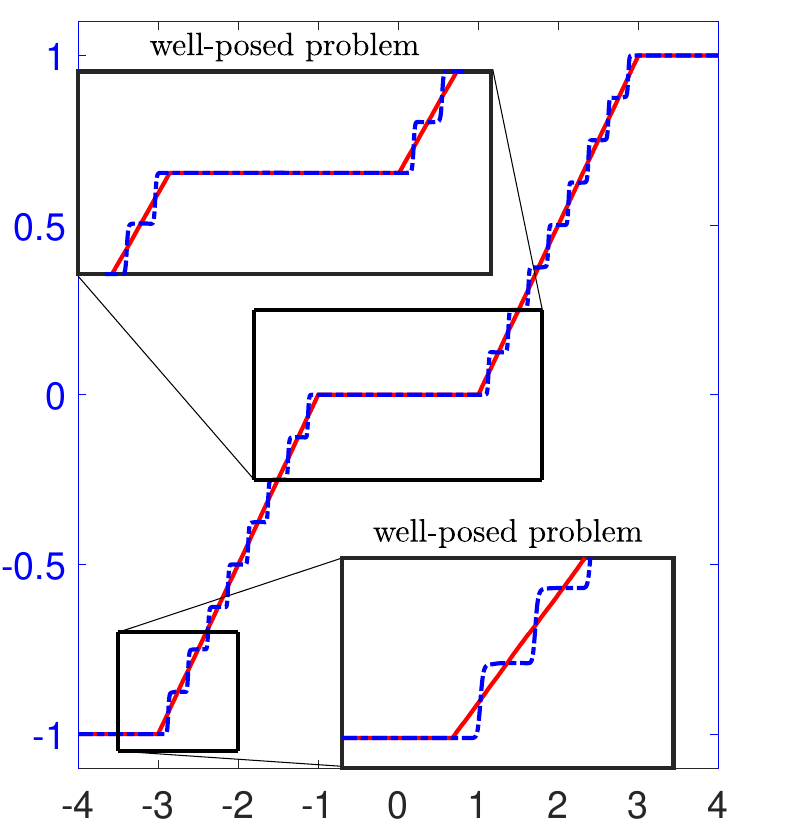}}
					
				\end{center}	
			\end{minipage}
			
			\caption{Simulation for one-dimensional level set equations in~${t=2}$ using normalized Legendre polynomials (left) and the Haar sequence with  $8$~basis functions. 
				Simulations in the left panel are obtained by the capacity-form differencing scheme from~\cite{LevelSet}, while the SSPRK-CWENO scheme is used for the Haar-type formulation. 		
				The solid red line states the mean of the random reference solution and the means that  correspond to the stochastic Galerkin systems are blue dotted. The zoom in the upper left 
				corner highlights the domain close to the random zero-level set. There, the optimization problem in Proposition~\ref{Lemma1} is ill-posed if Legendre polynomials are used, but it is well-posed for Haar-type expansions. 
				The zoom in the lower right corner shows a region where the optimization problem is well-posed.}
			\label{Example2Comparison}
		\end{figure}

	\end{remark}

\subsection{Gas flow with Lipschitz continuous pressure law}

This subsection is devoted to the \textit{p}-system with Lipschitz continuous flux function~\eqref{LeFloch} and  the corresponding stochastic Galerkin formulation that is  introduced in Corollary~\ref{Corollary3}. In particular, we consider the random pressure law
\begin{equation}\label{LeFlochRandomPressure}
	p(v,\xi)
	\coloneqq
	\begin{cases}
		v^{-\gamma_{1}} & \text{if \  } v<v_{*}(\xi), \\
		v^{-\gamma_{2}} 
		+\Delta v_{*}(\xi)
		& \text{if \  } v>v_{*}(\xi) \\
	\end{cases}
	\quad\text{with}\quad
	v_{*}(\xi)\sim \mathcal{U}[1,1.5].
\end{equation}


\begin{figure}[H]
	\begin{minipage}{0.49\textwidth}
		\begin{center}	
			\textbf{reference solution}
			
			\scalebox{1}{\includegraphics[width=\linewidth]{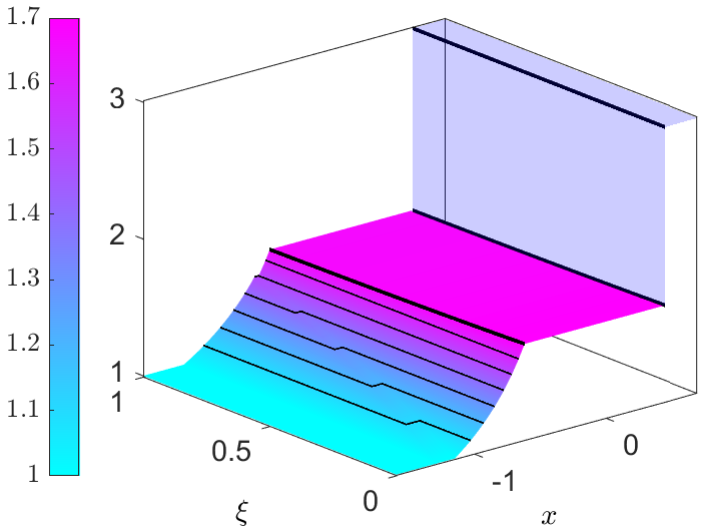}}
			
		\end{center}	
	\end{minipage}
	\hfil
	\begin{minipage}{0.49\textwidth}
		\begin{center}	
			
			\textbf{zoom on rarefaction wave}
			
			\scalebox{1}{\includegraphics[width=\linewidth]{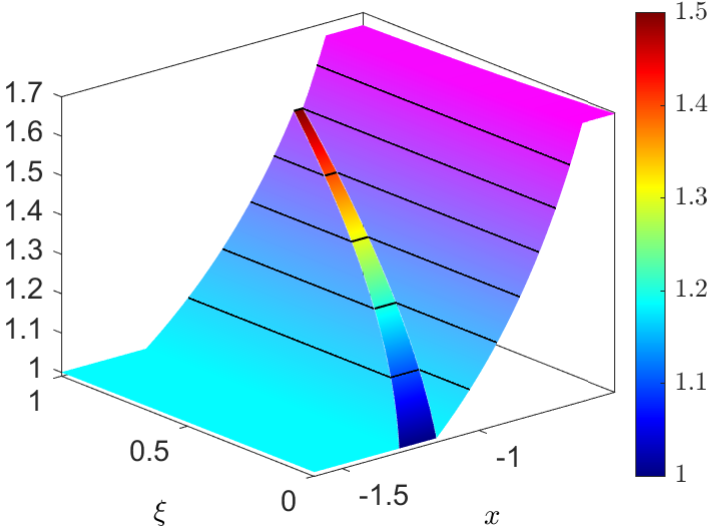}}
			
		\end{center}		
	\end{minipage}

	\vspace{6mm}

	\begin{minipage}{0.49\textwidth}
		\begin{center}	
			\textbf{gPC modes~$\boldsymbol{\widehat{v}^{(2)}(t,x)}$} 
		
		\scalebox{1}{\includegraphics[width=\linewidth]{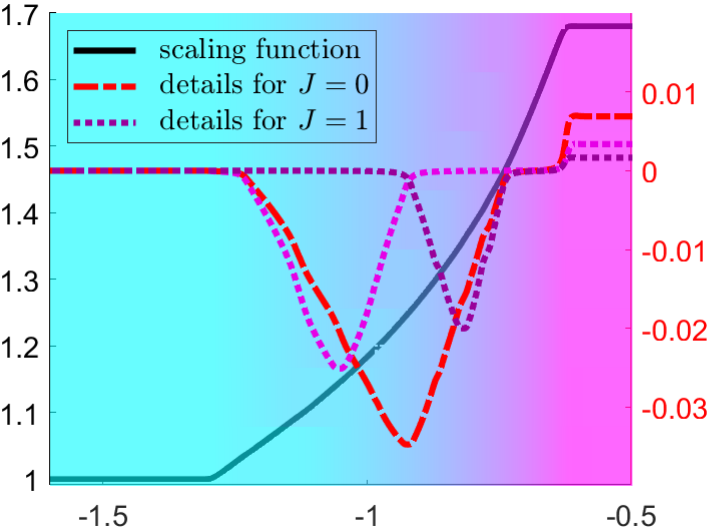}}
		
	\end{center}	
\end{minipage}
\hfil
\begin{minipage}{0.49\textwidth}
	\begin{center}	
		
		\textbf{realizations~$\boldsymbol{\Pi_J\big[ \widehat{v}^{(2)}(t,x) \big]\big(\xi(\omega) \big)}$}
		
		\scalebox{1}{\includegraphics[width=\linewidth]{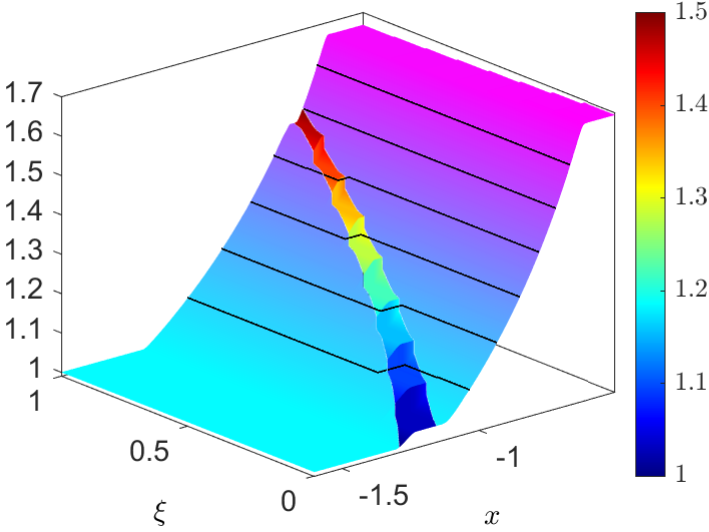}}
		
	\end{center}		
\end{minipage}

\caption{Upper panels show the reference solution to the hyperbolic system~\eqref{LeFloch} with random, Lipschitz continuous pressure law~\eqref{LeFlochRandomPressure} at time~$t=1$  with parameter~$\gamma_1=\nicefrac{5}{3}$ and $\gamma_2=\nicefrac{4}{3}$. Lower panels show the corresponding stochastic Galerkin formulation in Corollary~\ref{Corollary3} for the level~$J=2$. 
		Two different color schemes are used: The first colorbar, which is on the left side of the first subfigure and ranges from turquoise to pink, shows the values of the solution apart from the intermediate state in the rarefaction wave. This intermediate state itself is highlighted by  colorbars, which are on the right-hand side and range from blue to red.}
\label{Example3_Fig1}
\end{figure}


\begin{figure}[H]

\begin{center}	
	
	\textbf{intermediate state in rarefaction wave}
	
	\scalebox{1}{\includegraphics[width=\linewidth]{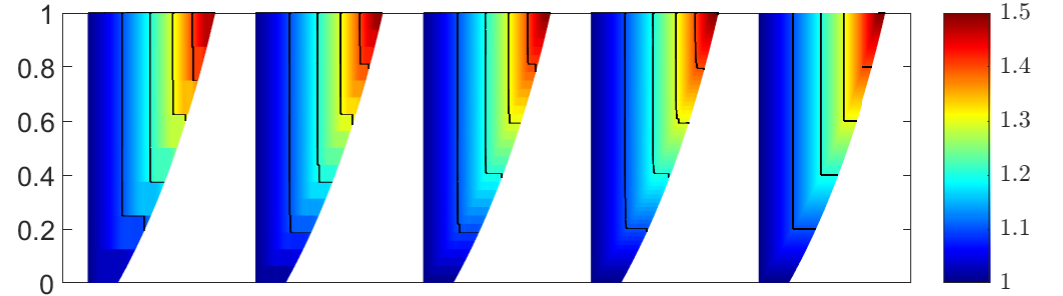}}
	
\end{center}

\hspace{8mm}
$J=2$
\hspace{7.6mm}
$J=3$
\hspace{7.6mm}
$J=4$
\hspace{7.6mm}
$J=5$
\hspace{7.6mm}
reference

\vspace{4mm}

\begin{center}	
	
	\textbf{sequence of scaled absolute differences}\vspace*{2mm}
	$\boldsymbol{
		2^{J-1} \,
		\Big\vert
		\Pi_{J} \big[ \widehat{v}^{(J)}(1,x) \big]
		-
		\Pi_{J-1}\big[ \widehat{v}^{(J-1)}(1,x) \big] 
		\Big\vert
		\big(\xi(\omega) \big)}$	
	
	\scalebox{1}{\includegraphics[width=\linewidth]{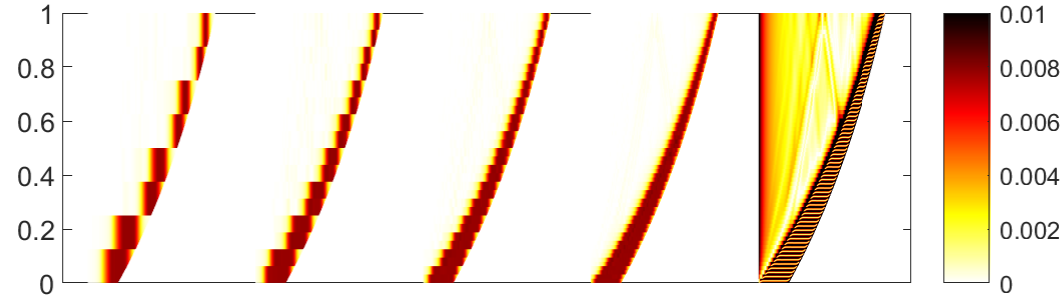}}
	
\end{center}

\hspace{8mm}
$J=2$
\hspace{7.6mm}
$J=3$
\hspace{7.6mm}
$J=4$
\hspace{7.6mm}
$J=5$
\hspace{7.6mm}
reference

\vspace{4mm}

\begin{center}	
	
	\textbf{absolute error to reference solution}\vspace*{2mm}
	$\boldsymbol{
		\Big\vert
		\Pi_{J} \big[ \widehat{v}^{(J)}(1,x) \big]
		-
		v^{\textup{(ref)}} (1,x) 
		\Big\vert
		\big(\xi(\omega) \big)}$

	\scalebox{1}{\includegraphics[width=\linewidth]{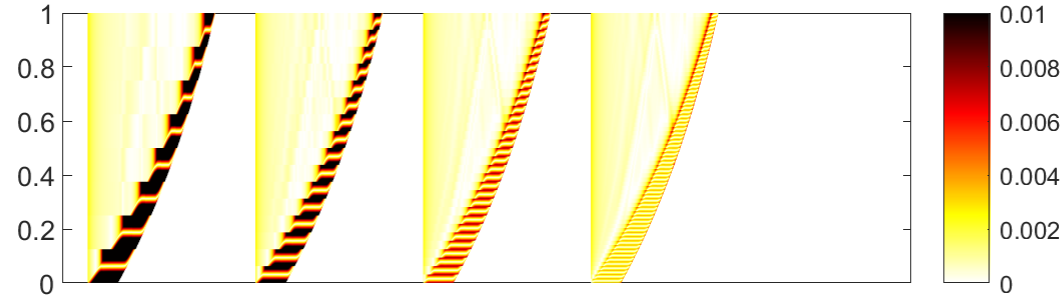}}

\end{center}

\hspace{8mm}
$J=2$
\hspace{7.6mm}
$J=3$
\hspace{7.6mm}
$J=4$
\hspace{7.6mm}
$J=5$

\vspace{4mm}

\caption{The upper panel shows a zoom on the intermediate state in the rarefaction wave for the levels~$J=2,\ldots,5$. Here, the stochastic Galerkin formulation of~Corollary~\ref{Corollary3}  and the cosine transform are used. The reference solution is given in~\cite[Sec.~5]{LipschitzFlux2001}.
	The middle  panel shows the absolute error of the stochastic Galerkin approximation between different levels that is scaled by the factor~$2^J$. 
	The lower panels states the absolute error to the reference solution.
}

\label{Example3_Fig2}

\end{figure}

\newpage
Deterministic initial values  are chosen  as~$\u(0,x)=(1,0)^\T$ 
for ${x<0}$ and~${\u(0,x) =(3,0)^\T}$ for ${x>0}$, respectively. As in the previous subsections, we consider the wavelet system~$\mathbb{W}\big[\VPhaar_{\cos}\big]$, which is generated by the discrete cosine transform, and we denote the gPC modes of the relative volume by~$\widehat{v}^{(J)}$, where the level~$J$ determines the cardinality~$\big\vert\widehat{v}^{(J)} \big\vert = 2^{J+1}$. 
The upper panels of Figure~\ref{Example3_Fig1} show the exact solution~of the relative volume~$v^{\textup{(ref)}}(t,x,\xi)$ that has been derived by Correia, Le~Floch and Thanh in~\cite[Sec.~5]{LipschitzFlux2001}. A rarefaction wave is connected by an intermediate state (purble) to a shock (transparent). More precisely, there is an intermediate constant state in the rarefaction wave, which results from a
discontinuity in the characteristic speeds. 
This constant state is highlighted in the right, upper panel  of Figure~\ref{Example3_Fig1}.

The lower, left panel shows the gPC modes~$\widehat{v}^{(2)}(1,x)$ at time~$t=1$ that are obtained by solving the intrusive system stated in Corollary~\ref{Corollary3} with level~$J=2$. The first mode~$\widehat{v}_0$, describing the mean of the random solution, is plotted with respect to the left axis. The details determine the variance and are stated at the right axis. 
The lower, right panel shows the realizations~$\Pi_K\big[ \widehat{v}^{(2)}(t,x) \big]\big(\xi(\omega) \big)$.  In particular, a stepwise approximation of the intermediate state in the rarefaction wave is observed. 


Figure~\ref{Example3_Fig2} is devoted to a convergence analysis for increasing levels~${J=2,\ldots,5}$. Therein, 
a  zoom in the area of the rarefaction wave is shown that arises from the right panels of Figure~\ref{Example3_Fig1}  and from a view at  the~$(x,\xi)$-plain. 
It  reveals the regularity and the accuracy of the stochastic Galerkin formulation in comparison to the reference solution, which is given on the right in the upper panel. The middle panel shows the absolute differences between the gPC approximation between the levels, which indicates the gain in accuaracy for increasing levels. This gain is weighted by the factor~$\sim 2^{J}$. Since this sequence of scaled absolute differences remains in the same magnitude, which is stated by the colorbar, the solution converges with the rate~$\nicefrac{1}{2}$. In particular, the simulation on the right in the middle panel states the absolute difference between the gPC approximation with level~$J=5$ and the reference solution, which is scaled by the factor~$2^5$. Therefore,  the approximation error is for the finest level in the magnitude~$0.01\cdot 2^{-5}$, as the right colorbar shows.  
The corresponding absolute error between the Galerkin formulations   and the reference solution is shown in the lower panel.

\subsection{Isentropic Euler equations}

The two-dimensional system~\eqref{EulerDeterministic} is considered with the choice~$\gamma=\nicefrac{4}{3}$. We use the stochastic Galerkin formulation that is derived in Corollary~\ref{corollary4} with the random initial values
$$\u(0,\x,\xi) 
=\begin{cases}
	\big(\rho_0(\xi),0\big)^\T & \text{for \ \ }
	\x\in[-1,1]^2 , \\
	\big(1,0\big)^\T & \text{else} 
\end{cases}
\qquad
\text{and}
\qquad
\rho_0(\xi)\sim\mathcal{U}[2,3].
$$
The left panel of Figure~\ref{Example4_Fig1} shows the mean of the random density. The  profile for $x_2=0$ is highlighted by a green line and analyzed in the right panel. Therein, a reference solution (blue, dotted) is shown that is obtained by a Monte Carlo simulation, where the solution to the random Riemann problem is exactly given in~\cite[Sec.~4]{Toro}. Furthermore, the bounds of the confidence region, where all exact random solutions occur, are shown as black lines. The grey region is the confidence region that is obtained by the intrusive formulation in~Corollary~\ref{corollary4}.

\begin{figure}[H]
	\begin{minipage}{0.49\textwidth}
		\begin{center}	
			\textbf{mean for intrusive formulation}
			
			\scalebox{1}{\includegraphics[width=\linewidth]{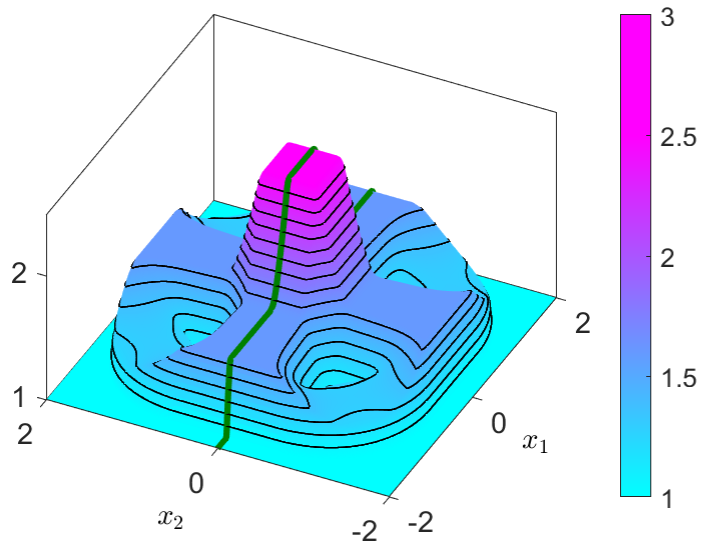}}
			
		\end{center}	
	\end{minipage}
	\hfil
	\begin{minipage}{0.49\textwidth}
		\begin{center}	
			\textbf{profile for~}$\boldsymbol{x_2=0}$
			
			\scalebox{1}{\includegraphics[width=\linewidth]{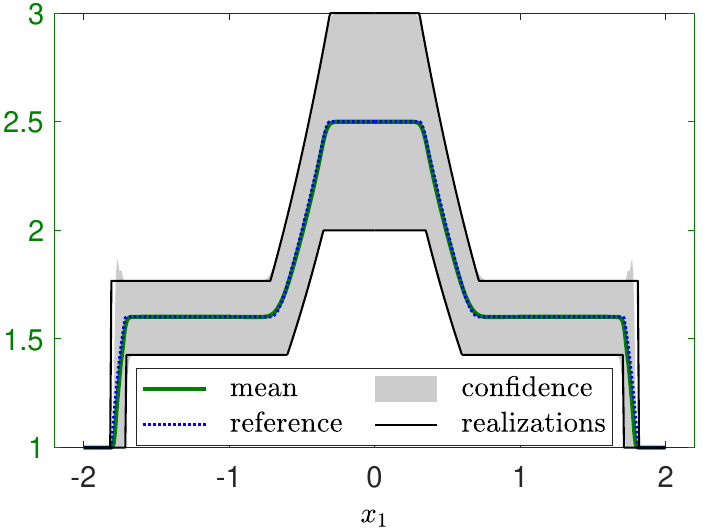}}
			
		\end{center}	
	\end{minipage}
	
	\caption{Left: Mean~$\widehat{\rho}_0(0.5,x)$ at time~$t=0.5$ to the intrusive formulation in Corollary~\ref{corollary4}. Right: Reference solution to the mean, which is illustrated   by the green profile in the left panel. }
	\label{Example4_Fig1}
\end{figure}

\section{Summary}
Stochastic Galerkin formulations have been investigated that are generated by Haar-type matrices. Those matrices arise in a number of  classical transforms, including the Haar and cosine transform, and generate a wavelet system, which states the functional dependence on a random state. This allows to express also non-polynomial quantities. In particular,  stochastic Galerkin formulations for  higher moments, the signum function and norms have been established. 

Furthermore, theoretical and numerical challenges are discussed that occur when classical polynomials are used in the polynomial chaos expansions. So far, it has been a major restriction in many applications to hyperbolic systems, when stochastic Galerkin projections not preserve deterministic  wellposedness results unless numerically expensive transforms are introduced. 
In contrast, the presented Haar-type formulations are stated in closed form and do not require any additional transforms. 
This allows for applications to a wide class of hyperbolic conservation laws. In particular, eigenvalue decompositions for several hyperbolic systems, including multi-dimensional level set equations and isentropic Euler equations with Lipschitz continuous flux, have been analyzed. 
The theoretical findings have been also numerically investigated by integrating the Galerkin system with globally third order one-dimensional and truly two-dimensional schemes which are based on CWENO reconstructions.

\section*{Appendix}

Stochastic Galerkin matrices enjoy for arbitrary basis functions  the 
properties~$\partial_{\widehat{\u}_j} \P(\widehat{\u}) = \M_j$ and~$\P(\widehat{\u})=\big[ \M_0 \widehat{\u} \big\vert \cdots \big\vert \M_K \widehat{\u} \big]$. Then, the  proof of Lemma~\ref{Lemma0} follows with standard algebraic arguments. 
\begin{proof}[{Proof of Lemma~\ref{Lemma0}:}]
	Because of the orthonormal eigenvalue decomposition, i.e.~$\VPk^\T(\widehat{\u})\VPl(\widehat{\u})=\delta_{k,\ell}$, the equality
	$$
	0
	= 
	\frac{1}{2}\,  \partial_{\widehat{\u}_j} \Big[ \VPk^\T(\widehat{\u}) \, \VPk(\widehat{\u}) \Big]
	=
	\VPk^\T(\widehat{\u}) \, \partial_{\widehat{\u}_j}  \VPk(\widehat{\u}) 
	$$
	holds, which implies
	\begin{equation}
		\VPk^\T(\widehat{\u}) \P(\widehat{\u})  \partial_{\widehat{\u}_j} \VPk(\widehat{\u})
		=
		\VPk^\T(\widehat{\u}) \DPk(\widehat{\u})  \partial_{\widehat{\u}_j} \VPk(\widehat{\u})
		=0.
		\label{Lemma2eq1}
	\end{equation}
	Due to~$\partial_{\widehat{\u}_j} \P(\widehat{\u}) = \M_j$, we have
	\begin{equation}
		\begin{aligned}
			0
			&=
			\VPk^\T(\widehat{\u})\;
			\partial_{\widehat{\u}_j}
			\Big[
			\P(\widehat{\u}) \VPk(\widehat{\u})
			-
			\DPk(\widehat{\u})
			\VPk(\widehat{\u})
			\Big]   \\
			&=
			\VPk^\T(\widehat{\u}) \Big[ \M_j- \partial_{\widehat{\u}_j} \DPk(\widehat{\u})\Big] \VPk(\widehat{\u})
			+
			\VPk^\T(\widehat{\u}) 
			\Big[
			\P(\widehat{\u})  
			-
			\DPk(\widehat{\u})
			\Big] 
			\partial_{\widehat{\u}_j} 
			\VPk(\widehat{\u}).  
		\end{aligned}
		\label{Lemma2eq2}
	\end{equation}
	Equations~\eqref{Lemma2eq1} and~\eqref{Lemma2eq2} yield
	\begin{equation}
		\partial_{\widehat{\u}_j} \DPk(\widehat{\u}) 
		=
		\VPk^\T(\widehat{\u}) \M_j \VPk(\widehat{\u}). 
		\label{Lemma2eq3}
	\end{equation}
	Then, the identity~$\P(\widehat{\u})=\big[ \M_0 \widehat{\u} \big\vert \cdots \big\vert \M_K \widehat{\u} \big]$ implies
	\begin{align}
		\partial_{\widehat{\u}_j}
		\DP(\widehat{\u}) \what
		&=
		\begin{pmatrix}
			v_0^\T(\widehat{\u}) \M_j v_0(\widehat{\u}) \what \\
			\vdots \\
			v_K^\T(\widehat{\u}) \M_j v_K(\widehat{\u}) \what
		\end{pmatrix}
		=
		\VP^\T(\widehat{\u}) \M_j \VP(\widehat{\u}) \what, \nonumber \\
		\D_{\bar{\alpha}}
		\DP(\bar{\alpha})\big\vert_{\bar{\alpha}=\widehat{\u}} \VP^\T(\uhat)\what
		&=
		\Big[
		\VP^\T(\uhat)
		\M_0 
		\what
		\Big\vert \cdots \Big\vert
		\VP^\T(\uhat)
		\M_K 
		\what 
		\Big]
		=
		\VP^\T(\uhat)
		\P(\what). \nonumber
	\end{align}
\end{proof}

\section*{Acknowledgments}
\noindent
Stephan Gerster and Giuseppe Visconti acknowledge the support of MUR (Ministry of University and Research) under the MUR-PRIN PNRR Project 2022 No.~P2022JC95T ``Data-driven discovery and control of multi-scale interacting artificial agent systems'' funded by European Union - Next Generation EU.\\

\noindent
Furthermore, we would like to offer special thanks to  \\ Matteo Semplice and Siegfried M\"uller.

\bibliographystyle{amsplain}
\bibliography{mybib}

\end{document}